\documentclass[a4paper,10pt,reqno]{amsart}

\usepackage{amsmath, amsthm, amsfonts, amssymb}
\usepackage[foot]{amsaddr}
\usepackage{mathrsfs}
\usepackage[shortlabels]{enumitem}
\usepackage{graphicx, booktabs, longtable}
\usepackage{subfigure}
\usepackage{makecell} 
\usepackage{caption}
\usepackage[dvipsnames]{xcolor}
\usepackage{tikz}
\usetikzlibrary{matrix}
\usepackage{mathtools} 
\usepackage[T2A, T1]{fontenc}  
\makeatletter
\def\easycyrsymbol#1{\mathord{\mathchoice
  {\mbox{\fontsize\tf@size\z@\usefont{T2A}{\rmdefault}{m}{n}#1}}
  {\mbox{\fontsize\tf@size\z@\usefont{T2A}{\rmdefault}{m}{n}#1}}
  {\mbox{\fontsize\sf@size\z@\usefont{T2A}{\rmdefault}{m}{n}#1}}
  {\mbox{\fontsize\ssf@size\z@\usefont{T2A}{\rmdefault}{m}{n}#1}}
}}
\makeatother
\newcommand{\Zhe}{\easycyrsymbol{\CYRZH}}

\usepackage{hyperref}
\usepackage{bm}

\newtheorem{theorem}{Theorem}[section]
\numberwithin{theorem}{section}
\newtheorem{prop}[theorem]{Proposition}
\newtheorem{lemma}[theorem]{Lemma}
\newtheorem{corollary}[theorem]{Corollary}
\theoremstyle{definition}
\newtheorem{definition}[theorem]{Definition}
\newtheorem{rem}[theorem]{Remark}
\newtheorem{example}[theorem]{Example}
\numberwithin{equation}{section}

\newcommand{\N}{\mathbb{N}}
\newcommand{\Z}{\mathbb{Z}}

\newcommand{\R}{\mathbb{R}}

\newcommand{\CC}{\mathbb{C}}

\newcommand{\sH}{\mathscr{H}}

\newcommand\e{\mathrm{e}}
\newcommand\I{\mathrm{i}}
\newcommand\re{\operatorname{Re}}

\newcommand{\rd}{\mathrm{d}}

\newcommand\dom{\operatorname{dom}}

\DeclareMathOperator{\Div}{div}
\DeclareMathOperator{\curl}{curl}

\newcommand\cQ{\mathcal Q}

\newcommand\cA{\mathcal A}

\newcommand\cD{\mathcal D}

\newcommand\cJ{\mathcal J}
\newcommand\cL{\mathcal L}
\newcommand\cM{\mathcal M}
\newcommand\cH{\mathcal H}

\newcommand\hd{\hat{\delta}}
\newcommand\hr{\hat{\rho}}

\newcommand\ov\overline

\newcommand\eps\varepsilon
\renewcommand\epsilon\varepsilon
\renewcommand\rho\varrho
\newcommand\al\alpha
\newcommand\la\lambda

\newcommand\ds\displaystyle

\newcommand\p\partial

\newcommand{\supp}{\operatorname{supp}}

\DeclarePairedDelimiter{\norma}{\lVert}{\rVert}
\newcommand{\ip}[2]{\left\langle #1,#2\right\rangle}
\newcommand{\dd}{\,\mathrm d}

\newcommand{\be}{\mathbf{e}}

\author[F.~Ferraresso]{Francesco Ferraresso}
\address{Department of Computer Science,
University of Verona, 
Strada Le Grazie 15, Ca' Vignal 2,
Verona 37134,
Italy }
\email{francesco.ferraresso@univr.it}

\author[M.~ Marletta]{Marco Marletta}
\address{School of Mathematics,
Cardiff University, Abacws,
Senghennydd Road, Cathays,
Cardiff CF24 4AG, UK}
\email{MarlettaM@cardiff.ac.uk}

\date{\today}

\thanks{}
\usepackage[update,prepend]{epstopdf}

\title[]{Essential spectral geometry of the Maxwell system in unbounded domains}

\begin{document}

\begin{abstract}
We analyse the essential spectrum of $\cM = \curl \curl$ acting on divergence-free vector fields in unbounded domains of $\R^3$. We show that $\sigma_e(\cM) = [0,+\infty)$ in quasi-conical domains and $\sigma_e(\cM) \neq \emptyset$ in quasi-cylindrical domains. For horn-shaped domains with circular cross-section and eventually mean-convex boundary, we establish that $\sigma_e(\cM) = \emptyset$, independently of their volume. For horns with annular cross-section, the transverse normal harmonic field determines an effective one-dimensional Schr\"odinger operator $H_V$ with $\sigma_e(\cM) \supseteq \sigma_e(H_V)$. Finally, for a concrete family of perforated exponential horns with a double-exponential hole, we show that depending on the rate of shrinking of the hole at infinity, either $\sigma_e(\cM) = \emptyset$, or $\sigma_e(\cM) = [\gamma^2, +\infty)$, or $\sigma_e(\cM) = [0,+\infty)$. 
\end{abstract}

\maketitle


\section{Introduction}

The essential spectrum of the Maxwell operator $\cM = \curl\curl$ in an unbounded region $\Omega\subseteq \R^3$ with perfectly conducting boundary conditions is partially determined by the geometry of $\Omega$ at infinity. As in the Laplacian case, regions that are `large' at infinity can support localised low frequency quasi-modes that spread out at infinity, generating a larger essential spectrum compared to regions such as waveguides with a confining geometry. By contrast, regions with fast shrinking ends {may} not support the construction of quasi-modes, preventing the loss of the compactness. 
{Beyond} domain `size', the Maxwell essential spectrum is also sensitive to topological effects: changing the cross-section of an infinite cylinder from a disk to an annulus induces a sudden change of the Maxwell essential spectrum, independently of the size of the hole. This paper establishes results relating the domain geometry and topology to the essential spectrum of $\cM$, with particular attention to the challenging case of filled or perforated horn-shaped domains.

We assume that the unbounded region $\Omega\subseteq \R^3$ is locally Lipschitz. With $X_{N0}(\Omega) := H_0(\curl, \Omega) \cap H(\Div 0, \Omega)$ (see Section \ref{sec:prelim} below for the definition of these spaces), define the sesquilinear form
\begin{equation}\label{intro:qM}
q_{\cM}(u, v) = \int_\Omega \curl u \cdot \overline{\curl v} \, dx, \quad u,v \in X_{N0}(\Omega).
\end{equation}
The associated quadratic form $q_{\cM}[u]$ is densely defined in the Hilbert space $\sH_0 = (H(\Div 0, \Omega), \norma{\cdot}_{L^2(\Omega)^3})$, and it is closed and nonnegative. By the second representation theorem \cite[Section VI.3.1]{Kato}, there exists a unique selfadjoint operator $\cM$ associated with $q_{\cM}[u]$, and the following equality holds 
\[
q_{\cM}(u, v) = (\cM^{1/2}u, \cM^{1/2} v), \quad u,v \in X_{N0}(\Omega).
\]
When $\Omega$ is smooth, $\cM$ can be equivalently be described by
\begin{equation}\label{intro:T}
\cM = \curl\curl, \quad \dom(\cM) = \{ u \in X_{N0}(\Omega) : \curl\curl u \in L^2(\Omega)^3 \}.
\end{equation}
The operator $\cM$ appears in the spectral analysis of the time-harmonic Maxwell system in homogeneous dielectric media, and in the analysis of the Hodge Laplacian acting on differential $1$-forms in three-dimensional Euclidean domains. A fundamental question is to relate the geometry of $\Omega$ at infinity with the essential spectrum of the operator $\cM$, where
\[
\sigma_e(\cM) = \big\{ \la \in \CC : \exists (x_n)_n \subset \dom(\cM) : \norma{x_n} = 1,\, x_n \rightharpoonup 0,\, \norma{(\cM - \la) x_n} \to 0 \big\}.
\]
Our first new results establish that: \\
(i) If $\Omega$ is quasi-conical, that is, it contains a sequence of disjoint arbitrarily large balls (see Fig. \ref{fig:quasiconical}), then $\sigma_e(\cM) = [0, +\infty)$. \\
(ii) If $\Omega$ is quasi-cylindrical, meaning that it is not quasi-conical and it contains a sequence of disjoint balls of fixed radius (see Fig. \ref{fig:quasicyl}), then $\sigma_e(\cM) \neq \emptyset$.\\

\begin{figure}
\centering
\includegraphics{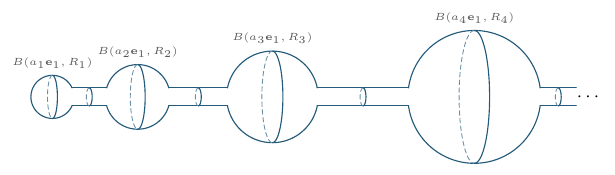}
\caption{A quasi-conical domain, obtained by gluing together a sequence of disjoint balls $B(a_j \mathbf{e}_1, R_j)$ with $a_j \to + \infty$, $R_j \to + \infty$}
\label{fig:quasiconical}
\end{figure}
\begin{figure}
\centering
\includegraphics{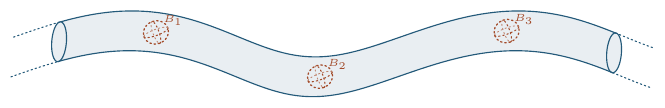}
\caption{A quasi-cylindrical domain, with a choice of disjoint balls $(B_j)_j$ of the same radius as in the definition.}
\label{fig:quasicyl}
\end{figure}

If $\Omega$ is quasi-bounded, that is, it is neither quasi-conical nor quasi-cylindrical, the situation is more complicated. 
We restrict our attention to horn-shaped domains 
\begin{equation}\label{intro:Horndom}
\Omega = \{ (\bar{x}, z) \in \R^2 \times \R : |\bar{x}| < \rho(z), z > 1 \},
\end{equation}
where $\rho \in C^{1,1}(1,\infty)$ is a positive, decreasing function with $\rho(z) \to 0$ as $z \to +\infty$, and with the property that
\begin{equation}\label{eq:rhoprime}
{\rm ess} \liminf_{z \to + \infty} (\rho'(z)^2 -\rho(z) \rho''(z)) > - 1.
\end{equation}
Note that \eqref{eq:rhoprime} is satisfied for the standard choices $\rho(z) = \e^{-z^\alpha}$, $\alpha > 0$, see Fig.\ref{fig:filledhorn} , and for $\rho(z) = z^{-\alpha}$, $\alpha > 0$. 
\begin{figure}[h!]
\centering
\includegraphics{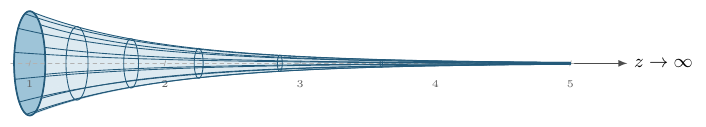}
\caption{The horn-shaped domain as in \eqref{intro:Horndom}, with $\rho(z) = \e^{-z}$.}
\label{fig:filledhorn}
\end{figure}

It is a fundamental result of Rellich, then generalised by Adams for arbitrary open sets satisfying the $C_{n,p}$-condition (see \cite{MR227765}), that $\sigma_e(-\Delta_{\Omega}^{\rm dir}) = \emptyset$. By contrast, for the Neumann Laplacian, results of Evans and Harris \cite{EvaHar} and of Davies and Simon \cite{DavSim} prove that $\sigma_e(-\Delta_{\Omega}^{\rm neu}) = \emptyset$  only provided that $\rho(z) \to 0$ fast enough as $z \to +\infty$; for instance, if $\rho(z) = \e^{-z^\alpha}$, $\alpha \in (0,+\infty)$, then $\alpha > 1$ is necessary and sufficient for $\sigma_e(-\Delta_{\Omega}^{\rm neu}) = \emptyset$. For the Maxwell operator the following theorem shows that, assuming \eqref{eq:rhoprime}, $\sigma_e(\cM)$ in a horn is empty, without additional constraints on the rate of shrinking at infinity.

\begin{theorem}\label{intro:thm:filledhorn}
Let $\Omega$ be as in \eqref{intro:Horndom}, with $\rho$ satisfying \eqref{eq:rhoprime}. Then $\sigma_e(\cM) = \emptyset$ and the embedding of $H_0(\curl, \Omega) \cap H(\Div, \Omega)$ into $L^2(\Omega)^3$ is compact. 
\end{theorem}
The proof of this theorem is based on an asymptotic weighted Maxwell-Poincar\'e inequality of independent interest, see Theorem \ref{thm:MaxGaff}. We remark that no assumptions on the measure of $\Omega$ are imposed. \\

We then consider the perforated horn (see Fig. \ref{fig:perforatedhorn})
\begin{equation}\label{intro:perforatedhorndeltarho}
\Omega_{\delta\rho} = \{ (\bar{x}, z) \in \R^2 \times \R : \delta(z) < |\bar{x}| < \rho(z), z > 1 \},
\end{equation}
with $C^{1,1}$ positive radii $0 < \delta(z) < \rho(z)$ decreasing to zero.

\begin{figure}[h!]
\centering
\includegraphics{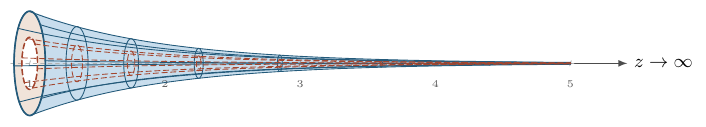}
\caption{The perforated horn \eqref{intro:Horndom}, with $\rho(z) = \e^{-z} = 2\delta(z)$}
\label{fig:perforatedhorn}
\end{figure}

As the boundary of $\Omega_{\delta\rho}$ has 2 connected components, it is well-known that the operator $\cM$ has a one-dimensional kernel given by a relative harmonic 1-form $h$, so $0 \in \sigma(\cM)$ in this case. However, establishing whether the presence of a nontrivial one-dimensional kernel affects $\sigma_e(\cM)$ is not easy. The key intuition here is that each fixed section $S_z = \{ \bar{x} \in \R^2 : \delta(z) < |\bar{x}| < \rho(z) \}$ supports the harmonic function $q(r,z) = \frac{\log(r) - \log(\delta(z))}{\log(\rho(z)) - \log(\delta(z))}$, whose gradient can be used as building block for a propagating quasi-mode for $\cM$. This idea gives a first estimate on $\sigma_e(\cM)$.
\begin{theorem}\label{thm:intro:perforatedhorn1}
Let $\Omega$ be as in \eqref{intro:perforatedhorndeltarho}. Define
\[
 M(z) = \log\bigg(\frac{\rho(z)}{\delta(z)}\bigg), \quad p(z) = \frac{M'(z)}{2 M(z)}, \quad V(z) = p^2(z) - p'(z),
\]
for $z > 1$, and the operator (as usual defined via its quadratic form)
\[
\begin{split}
&H_V = - \frac{\rd^2}{\rd z^2} + V, \\
&\dom(H^{1/2}_V) = \{ f \in L^2(1,\infty) : f \in H^1_{\rm loc}(1,\infty), f' + p f \in L^2(1,\infty), f(1) = 0 \}.
\end{split}
\]
With
\[
\hat{\delta}(z) = \frac{\delta'(z)}{\delta(z)}, \quad \hat{\rho}(z) = \frac{\rho'(z)}{\rho(z)}, 
\]
and 
\[
\eta(z) = \rho(z) \bigg( \int_0^1 (s \hat{\delta}(z) + (1-s) \hat{\rho}(z) )^2 e^{-2sM(z)} \, ds  \bigg)^{\frac12},
\]
for $z \in (1,\infty)$, assume that
\[
\lim_{Z \to + \infty} \sup_{z > Z} |\eta(z)| = 0.
\]
Then $\sigma_e(\cM) \supseteq \sigma_e(H_V)$. In particular, if $V(z) \to 0$ as $z \to +\infty$, then $\sigma_e(\cM) = [0,+\infty)$.
\end{theorem}

The proof Theorem \ref{thm:intro:perforatedhorn1} is based on the following mechanism: 1D Weyl singular sequences $(a_n)_n$ for $H_V$ can be transformed into 3D Weyl singular sequences for $\cM$ via the harmonic lifting $U_{a_n}(r, \theta, z) = a_n(z) \sqrt{\frac{M(z)}{2 \pi}} \nabla q(r,z)$, followed by the projection onto the divergence-free space $H(\Div 0, \Omega)$. The error introduced by this latter soleinodal projection is controlled by $\eta$, which corresponds to the $L^2$ norm of the axial component of $U_a$. 

Finally, for the specific choices of $\rho(z) = \e^{-z}$, $\delta(z) = \e^{-z-\e^{2\gamma z^\alpha}}$, with $\gamma$, $\alpha$ positive real parameters, we completely characterise the essential spectrum. 
\begin{theorem}\label{thm:sigmaessperforatedhornalpha}
Let $\Omega_{\delta\rho}$ be as in \eqref{intro:perforatedhorndeltarho}, with $\rho(z) = \e^{-z}$, $\delta(z) = \e^{-z-\e^{2\gamma z^\alpha}}$. Then
\[
\sigma_e(\cM) =
\begin{cases}
[0,+\infty), \quad &\textup{if $0 < \alpha < 1$,} \\
[\gamma^2, +\infty), \quad &\textup{if $\alpha = 1$,} \\
\emptyset, \quad &\textup{if $\alpha > 1$}.
\end{cases}
\] 
\end{theorem}
This result has been inspired by the analysis in \cite{EvaHar, DavSim} for the Neumann Laplacian, even though the mechanism that generates the different regimes for $\sigma_e(\cM)$ here is completely different. The proof of Theorem \ref{thm:sigmaessperforatedhornalpha} is based on an asymptotic Maxwell-Poincaré inequality, which holds up to small error terms as the support of the test vector fields moves out to $\infty$ along the horn. To establish such inequality, one needs a combination of different ideas and methods. After pullback via a suitable diffeomorphism, one reduces the analysis of $\cM$ in $\Omega_{\delta\rho}$ to the analysis of a second-order differential operator with non-constant coefficients acting in $L^2(C)^3$, where $C$ is the perforated cylinder $(1/2,1) \times \mathbb{S}^1 \times (1,\infty)$. However, unlike previously studied cases, e.g., \cite{MR3812006}, here the matrix-valued coefficients depend on all the variables. 

In order to handle the complexity of such coefficient interactions, instead of trying to construct directly approximate TE/TM/TEM modes, we first apply the 2D Hodge-Morrey decomposition in a fixed section $S_z$; and then we compare the sectional Maxwell energy ($\cQ_{\square}$), the energy of the operator $\cM$ compressed to ${\rm span}\{h\}$ ($q_p$), the energy of a suitably defined asymptotic Maxwell energy ($\cQ_0$) and the full Maxwell energy ($\cQ$). It turns out that the transverse TE and TM modes constructed in the section $S_z$ have arbitrarily large sectional energy as $z \to +\infty$, and therefore the contribution to the essential spectrum can only come from the TEM sectional mode, corresponding to a suitable normalization of the gradient of the harmonic function $q(r,z)$.

We finally remark that a combination of the proofs of Theorem \ref{thm:intro:perforatedhorn1} and Theorem \ref{intro:perforatedhorndeltarho} yield the following more general statement.

\begin{theorem} \label{thm:intro:generalperfhorn}
Let $\Omega_{\delta\rho}$ be as in \eqref{intro:perforatedhorndeltarho}. In addition to the assumptions of Theorem  \ref{thm:intro:perforatedhorn1}, assume that \eqref{eq:decayparameters}--\eqref{eq:decayparameters3} hold, where $\kappa$ and $\omega$ are defined as in Remark \ref{rem:generality}. Then $\sigma_e(\cM) = \sigma_e(H_V)$, i.e. the essential spectrum of the Maxwell operator coincides with the essential spectrum of a one-dimensional Schr\"{o}dinger operator.
\end{theorem}

This article is organized as follows. In Section \ref{sec:survey} we survey the available results on the essential spectrum of $\cM$ and of the scalar Laplacians in unbounded domains. Section \ref{sec:prelim} introduces the functional setting and the spectral theory framework. Section \ref{sec:quasiconcyl} treats quasi-conical domains and quasi-cylindrical domains. Section \ref{sec:filledhorn} is devoted to establishing Theorem \ref{intro:thm:filledhorn} for filled horns. Section \ref{section:7} deals with the perforated horn \eqref{intro:perforatedhorndeltarho}, and contains the proof of Theorem \ref{thm:intro:perforatedhorn1}, whereas Section \ref{sec:finalsec} contains the proof of the trichotomy result Theorem \ref{thm:sigmaessperforatedhornalpha}. The Appendix contains the proofs of some auxiliary results regarding the coordinate computations and estimates.

\section{A survey on the essential spectrum of the Laplacian and of Maxwell in unbounded domains}\label{sec:survey}

We shall give a brief overview of the main results regarding the spectrum of the Dirichlet and of the Neumann Laplacian in unbounded domains of $\R^N$, as some of the ideas and techniques available for the scalar case play a role for the spectral analysis of $\cM$ as well.

\subsection{The Dirichlet case}\label{subsec:Dirichlet}
Since the spectrum of the Dirichlet Laplacian $-\Delta^{\rm dir}$ enjoys the domain monotonicity property, it is expected that the essential spectrum $\sigma_e(-\Delta^{\rm dir})$ only depends on the "size of the domain at infinity". This leads to the following classification of open sets.

\begin{definition}\label{def: classification}
Let $\Omega \subset \R^N$ be an open set. We say that $\Omega$ is
\begin{itemize}
\item quasi-conical, if there exists a sequence of disjoint cubes $Q_n(d_n)$ of edge length $d_n$, with $d_n \to + \infty$, and $Q_n(d_n) \subset \Omega$ for all $n$. 
\item quasi-cylindrical, if it is not quasi-conical and there exists a sequence of disjoint cubes $Q_n(d)$ of fixed edge length $d > 0$, $Q_n(d) \subset \Omega$ for all $n$. 
\item quasi-bounded, if it is neither quasi-conical, nor quasi-cylindrical.
\end{itemize}
\end{definition}

Set $d_1 := \sup \{ d : \: \Omega \: \text{contains a sequence of disjoint cubes $\overline{Q}(d)$ of size $d$}\}$; then, $d_1 = + \infty$ for quasi-conical domains, $d_1 \in (0, +\infty)$ for quasi-cylindrical domains, and $d_1 = 0$ for quasi-bounded domains. 

For domains $\Omega$ that satisfy a uniform interior cone condition (UICC) (which imply that $\Omega$ is a locally finite union of uniformly Lipschitz sets, and it has strictly positive inradius) the classification of open sets in Definition \ref{def: classification} matches quite well the essential spectral properties of $-\Delta_\Omega^{\rm dir}$. Indeed, 
\begin{itemize}
\item $\Omega$ is quasi-conical and has the (UICC) if and only if $\sigma_e(-\Delta_\Omega^{\rm dir}) = [0,+\infty)$. 
\item $\Omega$ is quasi-cylindrical and has the (UICC) implies that $\sigma_e(-\Delta_\Omega^{\rm dir}) \neq \emptyset$ and $0 \notin \sigma_e(-\Delta_\Omega^{\rm dir})$.
\item $\Omega$ is quasi-bounded and has the (UICC) implies that $\Omega$ is bounded, hence $\sigma_e(-\Delta_\Omega^{\rm dir}) = \emptyset$. 
\end{itemize}

For less regular domains $\Omega$, the classification in Def. \ref{def: classification} alone gives only partial information on $\sigma_e(-\Delta_\Omega^{\rm dir})$, as removal of sets of zero capacity from $\Omega$ does not affect $\sigma_e(-\Delta_\Omega^{\rm dir})$, but may change $d_1$. Thus, an additional asymptotic capacity condition is required. We refer to \cite[Theorem 6.2]{EE}, \cite{MR2183285} for details about the essential spectrum of $-\Delta_\Omega^{\rm dir}$ in rough domains. 

An important class of quasi-cylindrical Lipschitz domains are bent and twisted waveguides. Their spectral theory is deeply interconnected with their geometry, 
see for instance \cite{MR1310767, MR2385742, MR2166298, MR1857842, MR1019002, MR3327061, MR2065680}. We mention also \cite{GJ92}, where not only the Dirichlet Laplacian, but also the Maxwell operator has been analysed in a few specific quasi-cylindrical geometries. Quasi-bounded horns provide a different class of examples, in which compactness
and eigenvalue asymptotics depend on the narrowing geometry. We refer to Adams \cite{MR227765} for compact Sobolev embeddings on unbounded domains, and to \cite{MR348295, MR757992} for eigenvalue asymptotics in horn-shaped regions, including regions of infinite volume.

\subsection{The Neumann case}\label{subsec:Neumann}
For the Neumann Laplacian, the connection between geometry and essential spectrum is more sensitive to the boundary.
%
In \cite{MR1140635}, it is shown that any closed subset of $[0,+\infty)$ may be the essential spectrum of $- \Delta_\Omega^{\rm neu}$, for a suitable choice of $\Omega$ among the open, connected subsets of the unit ball of $\R^N$. These bounded domains have a rough boundary, and indeed they are usually constructed as modifications of either the room-and-passages domain or of the comb domain (whose boundary fails to be the graph of a measurable function at least at one point). The other stark difference with the Dirichlet case is that domain monotonicity is not available anymore, and therefore a simple classification such as the one in Definition \ref{def: classification} (using only subdomains contained in $\Omega$) can only give partial information on $\sigma_e(-\Delta_\Omega^{\rm neu})$. 

It is expected that in an unbounded Lipschitz domain $\Omega$, the essential spectrum of $- \Delta_\Omega^{\rm neu}$ may be very large not only in quasi-conical domains, but even in some specific quasi-cylindrical domains. One can show indeed that if $\Omega$ is quasi-conical then $\sigma_e(- \Delta_\Omega^{\rm neu}) = [0, +\infty)$; but a direct computation shows that if $\Omega$ is a slab $\{(x_1, \dots, x_N) : |x_N | < h \}$ or a cylinder $\{(\bar{x}, x_N) : |\bar{x}| < R, x_N > 1 \}$, then again  $\sigma_e(- \Delta_\Omega^{\rm neu}) = [0, +\infty)$. Moreover, unbounded Lipschitz domains $\Omega$ with $\sigma_e(-\Delta_\Omega^{\rm neu}) = \emptyset$ are rather rare. Adams \cite[Corollary 6.46]{AdamsFou} shows that if $|\Omega \cap B(0,r)^C|$ does not decay superexponentially as $r \to + \infty$, then $\sigma_e(- \Delta_\Omega^{\rm neu}) \neq \emptyset$. 

There is also a very elegant theory developed by Amick \cite{MR502660}, Evans-Harris \cite{EvaHar}, and Edmunds-Evans \cite{EE79} connecting the so-called Amick constant, the measure of non-compactness of the embedding $\iota : H^1(\Omega) \to L^2(\Omega)$ and the reduced minimum modulus of $- \Delta_\Omega^{\rm neu}$. However, as far as we are aware, estimates of these quantities in terms of concrete geometrical quantities are not available for general domains. Among other results for the class of generalised ridged domains, in the case of the horn $\Omega = \{ (\bar{x}, z) \in \R^N : z > 0, |\bar{x}| < \rho(z)\}$, \cite{EvaHar} shows that $0 \in \sigma_e(-\Delta_\Omega^{\rm neu})$ if and only if
\[
\text{either} \:\: \int_0^\infty \rho(r)^{N-1} dr = \infty \quad \text{or} \quad \limsup_{s\to +\infty} \bigg( \int_s^\infty \rho(r)^{N-1} \, dr \int_0^s \rho(r)^{1-N} dr \bigg)= \infty.
\]
%

\subsection{The Maxwell case}
The analysis of the essential spectrum of the operator $\cM = \curl\curl$ when $\Omega$ is unbounded has not been considered in detail in the literature. In the last 7 years, starting from \cite{MR3942228}, a new rising interest in these problems has been observed, see e.g., \cite{MR3812006, MR4057880, MR4774282, BCOZ25, dhia2026}. 

First of all, the spectrum of $\cM$ is explicitly computable in some specific geometries \cite{MR3942228}: 
\begin{itemize}
\item If $\Omega = \R^3$, then $\sigma(\cM) = \sigma_e(\cM) = [0,+\infty)$. 
\item Given $h > 0$, if $\Omega = \{ (\bar{x}, z) \in \R^3 : |z| < h/2 \}$, then $\sigma(\cM) = \sigma_e(\cM) = [0,+\infty)$.
\item If $\Omega$ is an infinite cylinder with section $(0, L_2) \times (0, L_3)$, then $\sigma_e(\cM) = [\frac{\pi^2}{\max\{L_2^2, L_3^2\}}, + \infty)$.
\item In bounded Lipschitz domains, $\sigma_e(\cM) = \emptyset$.
\end{itemize}

In the straight cylinder $\Omega = D \times \R$, Filonov \cite{MR3812006, MR4057880, MR4774282} treats the case of non-constant-coefficient Maxwell operators. In particular, one rediscovers the classical result (see e.g., \cite{MR4398317} ) that
\begin{equation}\label{sigmaecylinder}
\sigma_e(\cM) = 
\begin{cases}
[\mu_2(D), + \infty), \quad &\textup{if $D$ is simply connected,} \\
[0, +\infty), \quad &\textup{otherwise}.
\end{cases}
\end{equation}
Here $\mu_2(D)$ is the second Neumann Laplacian eigenvalue in the cross-section $D$. The result \eqref{sigmaecylinder} can be further generalised to domains with a finite number of (straight) cylindrical ends \cite{MR4703376}.

Building on the previous result by Filonov, in \cite{BCOZ25} the authors proved that for the specific class of bent waveguides $\Omega$ the essential spectrum of $\cM$ in $\Omega$ coincide with the essential spectrum of $\cM$ in the straight cylinder $D \times \R$, provided that the curvature $\kappa$ of the axial curve tends to zero at infinity. 

Finally, \cite{PW26} recently treated the existence of a Gaffney inequality in unbounded domains. Among other results they show that if $\Omega$ is the bilipschitz image of an unbounded convex set $C$, then $0 \notin \sigma_e(\cM)$ if and only if $\Omega$ is bounded in (at least) 2 directions. Among the examples they consider an infinite L-shaped pipe and an infinite spiraling tube, and they show that $0 \notin \sigma_e(\cM)$. 

The representation of $\cM$ as the relative Hodge laplacian on co-closed 1-forms places the Maxwell operator within the literature on differential forms. Relevant examples include \cite{MR2087228, MR2274831, MR2833575, MR1047764} on noncompact manifolds \emph{without boundary} and \cite{MR2020035} for compact manifolds with or without boundary. We remark that the spectral analysis of $\cM$ in unbounded domains further requires the control of the boundary conditions even in domains whose boundary degenerates at infinity.

We conclude this section with a table comparing the essential spectra of $-\Delta_\Omega^{\rm dir}$, $-\Delta_\Omega^{\rm neu}$ and of $\cM$ in concrete cases. We use the wording `exponential horn' to refer to the prototypical case in which $\rho(z) = \e^{-z}$. We remark that, as far as we are aware, the essential spectrum of $-\Delta^{\rm neu}$ in a general perforated exponential horn is not known.

\begin{table}[htbp]
\centering\small
\setlength{\tabcolsep}{4pt}
\renewcommand{\arraystretch}{1.45}
\begin{tabular}{@{}p{43mm}ccc@{}}
\toprule
Domain & $\sigma_e(-\Delta^{\rm dir})$
& $\sigma_e(-\Delta^{\rm neu})$ & $\sigma_e(\cM)$\\
\midrule
$\R^3$ & $[0,\infty)$ & $[0,\infty)$ & $[0,\infty)$\\
$\R^2\times(-h/2,h/2)$ & $[\pi^2/h^2,\infty)$ & $[0,\infty)$ & $[0,\infty)$\\
$(0,L)^2\times\R$ & $[2\pi^2/L^2,\infty)$ & $[0,\infty)$ & $[\pi^2/L^2,\infty)$\\
$A(r,R)\times\R$ & $[\lambda_1(A(r,R)),\infty)$ & $[0,\infty)$ & $[0,\infty)$\\
Filled exponential horn & $\varnothing$ & $[1,\infty)$ & $\varnothing$  (Thm.\ref{intro:thm:filledhorn})\\
Perforated exponential horn & $\varnothing$ & Not addressed & See Thm. \ref{thm:sigmaessperforatedhornalpha}\\
\bottomrule
\end{tabular}
\caption{Essential spectra for the Dirichlet and Neumann Laplacians, and for Maxwell, in model unbounded domains of $\R^3$.}
\end{table}

\section{Preliminaries and notation}\label{sec:prelim}

\subsection{Function spaces and operators}
Let $\Omega$ be a locally Lipschitz connected open set of $\R^3$, with outer unit normal $\nu$. We will work either in full Hilbert space $L^2(\Omega)^3$, with associated norm $\norma{\cdot}$, or in the divergence-free Hilbert space $\sH_0 = (H(\Div 0, \Omega), \norma{\cdot})$, where
\[
H(\Div 0, \Omega) = \{ u \in L^2(\Omega)^3 : (u, \nabla \xi ) = 0, \quad \xi \in \dot{H}^1_0(\Omega) \},
\]
where 
$$\dot{H}^1_0(\Omega) = \{ \varphi \in L^6(\Omega) : \nabla \varphi \in L^2(\Omega), \: \varphi = 0\: \text{locally on $\p \Omega$} \}$$ 
is the homogeneous Sobolev space, endowed with the norm $\norma{\nabla \cdot}_{L^2(\Omega)^3}$. Note that $\dot{H}^1_0(\Omega)$ coincides with the usual $H^1_0(\Omega)$ whenever a Poincaré inequality for functions having square integrable gradient is available in $\Omega$. $\sH_0$ is a well-defined Hilbert space in view of the Helmholtz decomposition
\[
L^2(\Omega)^3 = \nabla \dot{H}^1_0(\Omega) \oplus H(\Div 0, \Omega).
\]
Define the Sobolev spaces 
\begin{align}
&H(\Div, \Omega) = \{ u \in L^2(\Omega)^3 : \Div u \in L^2(\Omega) \}, \quad \norma{u}^2_{H(\Div, \Omega)} = \norma{u}^2 + \norma{\Div u}^2, \\
&H(\curl, \Omega) = \{ u \in L^2(\Omega)^3 : \curl u \in L^2(\Omega)^3 \}, \quad \norma{u}^2_{H(\curl, \Omega)} = \norma{u}^2 + \norma{\curl u}^2.
\end{align}
Further define 
\[
H_0(\curl, \Omega) := \overline{C^\infty_c(\Omega)^3}^{H(\curl, \Omega)}.
\]
It is a classical result that if $\Omega$ is locally Lipschitz, then $u \in H(\curl, \Omega)$ admits a unique well-defined distributional tangential trace $\nu \times u$, in suitable defined space of distributions of regularity $H^{-1/2}$, see \cite{MR1944792}. Thus, $H_0(\curl, \Omega)$ admits the equivalent reformulation
\[
H_0(\curl, \Omega) = \{ u \in H(\curl, \Omega) : \nu \times u = 0 \: \textup{on $\p \Omega$} \}.
\]
We now introduce the spaces
\begin{align}
&X_N(\Omega) = H_0(\curl, \Omega) \cap H(\Div, \Omega), \quad \norma{u}^2_{X_N(\Omega)} = \norma{\curl u}^2 + \norma{\Div u}^2 + \norma{u}^2, \\
&X_{N0}(\Omega) = H_0(\curl, \Omega) \cap H(\Div 0, \Omega), \\
&H^1_N(\Omega) = \{ u \in H^1(\Omega)^3 : \nu \times u = 0, \: \textup{on $\p \Omega$} \}.
\end{align}

\begin{rem}
As a warning for the reader, we emphasize that 
\[
\overline{X_N(\Omega)}^{L^2(\Omega)^3} = L^2(\Omega)^3, \quad \overline{X_{N0}(\Omega)}^{L^2(\Omega)^3} = H(\Div 0, \Omega),
\]
and by definition $ \overline{C^\infty_c(\Omega)^3}^{H(\curl, \Omega)} = H_0(\curl, \Omega)$, but
\[
\overline{C^\infty_c(\Omega)^3 \cap X_N(\Omega)}^{X_N(\Omega)} \neq X_N(\Omega), \quad \overline{C^\infty_c(\Omega)^3 \cap X_{N0}(\Omega)}^{X_N(\Omega)} \neq X_{N0}(\Omega).
\]
\end{rem}

We recall that in view of a result by Weck \cite{MR343771} and Weber \cite{MR561375}, the space $X_N(K)$ is compactly embedded in $L^2(K)^3$ whenever $K$ is a bounded Lipschitz domain of $\R^3$. In particular, the space $X_{N}$ and its closed subspace $X_{N0}$ are locally compactly embedded in $L^2$. \\
Recall that $\cM$ is defined via its quadratic form $q_{\cM}$, as in \eqref{intro:T}. It is worth noting that $q_{\cM}[u] = 0$ is possible whenever $\p \Omega$ is not connected. If $\p \Omega = \bigcup_{j=1}^M \Gamma_j $, then there exists $M-1$ (non-constant) harmonic functions $\varphi_j$ in $\Omega$, with $\varphi_j = 1$ on $\Gamma_j$ and $\varphi_j = 0$ on $\bigcup_{k \neq j} \Gamma_k$. The vector fields $h_j = \nabla \varphi_j$ therefore satisfies $\curl h_j = \Div h_j = 0$ in $\Omega$, $\nu \times h_j = 0$ on $\p \Omega$, and assuming that $h_j \in L^2(\Omega)^3$, $h_j \in X_{N0}(\Omega)$, so $q_{\cM}[h_j] = 0$ for all $j = 1, \dots, M -1$. The 1-form corresponding to $h_j$ via the usual musical isomorphism is called relative harmonic 1-form, and the dimension of the kernel $\ker q_{\cM}$ coincides with $M -1$, the number of connected components of $\p \Omega$ minus 1.

A closely related operator is $\Delta_{\rm Hod}^{\rm rel}$ that, on smooth domains $\Omega$, acts as the differential expression $\curl \curl - \nabla \Div$ on the domain
\[
\dom(\Delta_{\rm Hod}^{\rm rel}) = \{ u \in X_N(\Omega) : \curl \curl - \nabla \Div \in L^2(\Omega)^3, \, \Div u = 0 \, \textup{on $\p \Omega$} \, \}.
\]
This is the so-called relative Hodge Laplacian acting on $1$-forms $L^2 \Lambda^1(\Omega)$. It can be equivalently be written as $\Delta_{\rm Hod}^{\rm rel} = \delta d + d \delta$, where $d$ is the exterior derivative and $\delta$ is the co-differential $\delta = -\star d \star$, $\star$ being the Hodge duality operator. In bounded Lipschitz domains $\Omega$, it is a well-known fact that the (discrete) spectrum of $\Delta_{\rm Hod}^{\rm rel}$ is the union of the Dirichlet Laplacian eigenvalues and the (discrete) spectrum of $\cM$ in $\Omega$.

\subsection{Gaffney inequality and Reilly identity}
Let $\Omega$ be a bounded domain of $\R^3$, and assume that either $\Omega$ is convex or $\p \Omega$ is of class $C^{1,1}$. An important result due to Gaffney \cite{MR48138} and Friedrichs \cite{MR87763}, valid in the more general Riemannian setting, states that there exists a constant $C > 0$ such that
\begin{equation}\label{eq:Gaff}
\norma{D u}^2 \leq C \big( \norma{\curl u}^2 + \norma{ \Div u}^2 + \norma{u}^2),
\end{equation}
for all $u \in X_N(\Omega)$. In particular, the space $X_N(\Omega)$ is continuously embedded in $H^1_N(\Omega)$. Another fundamental result is the so-called Reilly identity \cite{MR474149}, first proved in the context of Riemannian geometry (see also the related result by Grisvard \cite[Thm 3.1.1.1]{MR3396210} in Euclidean domains). In the case of vector fields that are normal to $\p \Omega$, Reilly identity states that if $\p \Omega$ is of class $C^{1,1}$, then
\[
\int_\Omega |D u|^2 = \int_\Omega (|\curl u|^2 + |\Div u|^2 ) - 2 \int_{\p \Omega} \cH (u \cdot \nu)^2\, d\sigma,
\]
for all vector fields $u \in H^1_N(\Omega)$, where $\cH$ is the mean curvature of $\p \Omega$. We immediately conclude that if $\cH \geq 0$, that is, if $\Omega$ is mean-convex, and $\p \Omega \in C^{1,1}$, then 
\begin{equation}\label{eq:strongGaff}
\int_\Omega |D u|^2 \leq \int_\Omega (|\curl u|^2 + |\Div u|^2 ),
\end{equation}
for all $u \in H^1_N(\Omega)$, hence for all $u \in X_N(\Omega)$ by density. Thus, in sufficiently regular mean-convex domains $\Omega$, the strong Gaffney inequality \eqref{eq:strongGaff} holds.
We will see that this will play an important role in unbounded domains that have asymptotically positive mean curvature. 

\subsection{Spectral theory tools} \label{subsec:specth}
Let $\cL$ be a densely defined linear operator in the separable Hilbert space $\sH$, with domain $\dom(\cL) \subset \sH$. We define the essential numerical range of $\cL$ (see \cite{MR4083777}) by
\[
W_e(\cL) = \big\{ \la \in \CC : \exists (x_n)_n \subset \dom(\cL) : \norma{x_n} = 1,\, x_n \rightharpoonup 0,\, (\cL x_n, x_n) \to \la \big\}.
\]
The set $W_e(\cL)$ is always closed and convex, and it contains ${\rm conv}(\sigma_e(\cL))$, the convex hull of the (Weyl) essential spectrum
\[
\sigma_e(\cL) = \big\{ \la \in \CC : \exists (x_n)_n \subset \dom(\cL) : \norma{x_n} = 1,\, x_n \rightharpoonup 0,\, \norma{(\cL - \la) x_n} \to 0 \big\}.
\]
The fundamental property that we will use in the sequel (see \cite[Thm. 3.8]{MR4083777}) is that, if $\cL$ is selfadjoint and semibounded from below, then in fact
\[
W_e(\cL) = {\rm conv}(\hat{\sigma}_e(\cL)), \quad \hat{\sigma}_e(\cL) = \sigma_e(\cL) \cup \{+\infty\}.
\]

In addition to $\sigma_e(\cL)$ and $W_e(\cL)$, assuming that $\sH = L^2(\Omega)^3$ for an unbounded Lipschitz domain $\Omega$ of $\R^3$, we need to define the Zhislin version of these sets, recently investigated in \cite{SBMMCT}. Assume that $(R_n)_n$ is any sequence of nondecreasing positive real numbers with $R_n \to + \infty$ as $n \to + \infty$. Then we define
\[
\begin{split}
\sigma_{e, \Zhe}(\cL) &= \{ \omega \in \CC : \: \textup{there exists $(u_n)_n \subset \dom(\cL)$, $\norma{u_n} = 1$} \\
&\textup{$ \supp u_n \subset \Omega\setminus\overline{B(0,R)} $, and $\cL u_n - \omega u_n \to 0$ in $L^2(\Omega)^3$ } \},
\end{split}
\]
and the Zhislin essential numerical range \cite{SBMMCT}
\[
\begin{split}
W_{e, \Zhe}(\cL) &= \{ \omega \in \CC : \: \textup{there exists $(u_n)_n \subset \dom(\cL)$, $\norma{u_n} = 1$} \\
&\textup{$ \supp u_n \subset \Omega\setminus\overline{B(0,R)} $, and $(\cL u_n, u_n) - \omega \to 0$ as $n \to +\infty$} \}.
\end{split}
\]

We recall the following useful property of the Zhislin essential numerical range.

\begin{lemma}\label{lemma:convexZENR}
Let $\cL$ be any densely defined, closed linear operator with domain $\dom(\cL) \subset L^2(\Omega)^3$, for some locally Lipschitz open set $\Omega$ of $\R^3$. 
The Zhislin essential numerical range $W_{e,{\scriptscriptstyle \Zhe}}(\cL)$ is convex.
\end{lemma}
\begin{proof}
Suppose that $(u_n)_{n\in\mathbb N}$ and $(v_n)_{n\in\mathbb N}$ are Zhislin
sequences such that $\lambda_n := (\cL u_n, u_n)  \to \lambda$ and $\mu_n:= (\cL v_n , v_n) \to \mu$.
Consider the $2\times 2$ matrices
\[ M_n = \left(\begin{array}{cc} (\cL u_n, u_n) & (\cL u_n, v_n) \\
 (\cL v_n, u_n) & (\cL v_n, v_n) \end{array}\right), \]
 observing that both $\lambda_n$ and $\mu_n$ lie in the numerical range $W(M_n)$ which is a convex set.
Consequently, given any point $\omega_n \in\mbox{conv}\{\lambda_n,\mu_n\}$ there exists $\theta_n\in [0,\pi)$
such that, with $w_n = \cos(\theta_n)u_n + \sin(\theta_n)v_n$,
\[ \frac{(\cL w_n, w_n)}{\| w_n \|^2} = \omega_n; \]
since $\lambda_n\to\lambda$ and $\mu_n\to\mu$, given any point $\omega \in\mbox{conv}\{\lambda,\mu\}$ 
we may choose the $\omega_n$ to be such that $\omega_n\to \omega$, $n\to\infty$.
By extracting a subsequence of $(v_n)_{n\in\mathbb N}$ we may assume without loss of generality
that $(u_n,v_n)\to 0$ as $n\to\infty$, and hence $\| w_n \| \to 1$ as $n\to \infty$. Thus $\tilde{w}_n:=w_n/\| w_n\|$
gives a Zhislin sequence such that $ (\cL \tilde{w}_n, \tilde{w}_n) = \omega_n \to \omega$, $n\to\infty$. \end{proof}

\begin{rem}
From \cite[Prop. 4.5(ii)]{SBMMCT}, we further deduce that if $\cL$ is a selfadjoint and semibounded operator with associated quadratic form $q_{\cL}$, the Zhislin essential spectrum of $q_{\cL}$, defined as
\[
\begin{split}
W_{e, \Zhe}(q_{\cL}) &= \{ \omega \in \CC : \: \textup{there exists $(u_n)_n \subset \dom(q_{\cL})$, $\norma{u_n} = 1$} \\
&\textup{$ \supp u_n \subset \Omega \setminus \overline{B(0,R_n)}$, and $q_{\cL}[u_n] \to \omega$ as $n \to +\infty$} \}
\end{split}
\]
coincides with $W_{e, \Zhe}(\cL)$. 
\end{rem}

\section{Quasi-conical and quasi-cylindrical domains}\label{sec:quasiconcyl}

\subsection{Quasi-conical domains} We consider here quasi-conical domains $\Omega$ in the sense of Definition \ref{def: classification}. We remark that in the following lemma we do not require any regularity of $\p \Omega$.
%
\begin{lemma}\label{lemma:quasi-conical}
If $\Omega$ is a quasi-conical open set of $\R^3$, then $\sigma_e(\cM) = [0,+\infty)$.
\end{lemma}
\begin{proof}
Since $\Omega$ is quasi-conical, there exists a sequence of cubes $Q_m = Q(x_m, \rho_m] \subset \Omega$ with $\rho_m \to + \infty$ as $m \to + \infty$. Let $\psi \in C^\infty((-1,1), \R)$ be a smooth function with compact support, $\norma{\psi}_{L^2(-1,1)} = 1$, and define $\psi^i_m(x_i) = \frac{1}{\sqrt{\rho_m}} \psi(\frac{x_i - (x_m)_i}{\rho_m})$. With $x = (\bar{x}, x_3)$, define $\psi_m(x) = \psi^1_m(x_1) \psi^2_m(x_2) \psi^3_m(x_3)$. Note that $\supp \psi_m \subset Q_m$ and $\norma{\psi_m}_{L^2(\Omega)} = 1$. Further define 
\[E_m(x) =\frac{1}{|\bar{k}|}\curl (\e^{\I \bar{k} \cdot \bar{x}} \psi_m(x)\be_3) = \frac{1}{|\bar{k}|}(\nabla (\e^{\I \bar{k} \cdot \bar{x}} \psi_m(x)) \times \be_3)\] 
for $\bar{k}\in \R^2 \setminus \{0\}$, extended-by-zero to the whole of $\Omega$. We claim that $(E_m)_m$ is a Weyl singular sequence, verifying $E_m \in \dom(\cM)$, $\norma{E_m}_{L^2(\Omega)} = 1 + O(1/\rho_m)$, $E_m \rightharpoonup 0$ and $\curl^2 E_m - |\bar{k}|^2 E_m \to 0$ as $m \to + \infty$. \\
We first note that $E_m$ is a smooth function, $\Div E_m = 0$ by definition, and
\[
E_m(x) = \frac{1}{|\bar{k}|}\begin{pmatrix}
-\p_{x_2} (\e^{\I \bar{k} \cdot \bar{x}} \psi_m(x))\\
\p_{x_1} (\e^{\I \bar{k} \cdot \bar{x}} \psi_m(x))\\
0
\end{pmatrix}
=: \frac{1}{|\bar{k}|} ((\nabla_{\bar{x}}^\perp \e^{\I \bar{k} \cdot \bar{x}}) \psi_m(x) + \e^{\I \bar{k} \cdot \bar{x}} \nabla_{\bar{x}}^\perp \psi_m(x))
\]
Therefore, $\supp E_m \subset Q_m$ for all $m$, hence it is clear that $E_m \rightharpoonup 0$ in $L^2(\Omega)$. We further note that $\norma{\p_i \psi_m}_{L^2(\Omega)} = O(1/\rho_m)$ as $m \to + \infty$, $i = 1,2,3$; we conclude that 
\[
\norma{E_m}_{L^2(\Omega)} = 1 + O(1/\rho_m)
\]
as $m \to + \infty$. Finally, since $\curl \curl E_m = - \Delta E_m$ as $\Div E_m = 0$, 
\begin{multline*}
\curl \curl E_m - |\bar{k}|^2 E_m = \frac{1}{|\bar{k}|} ((\nabla_{\bar{x}}^\perp (-\Delta - |\bar{k}|^2)\e^{\I \bar{k} \cdot \bar{x}}) \psi_m(x) \\
+ 2 (\nabla_{\bar{x}}^\perp (\nabla_{\bar{x}}  \e^{\I \bar{k} \cdot \bar{x}})\cdot \nabla_{\bar{x}} \psi_m)+ O(1/\rho_m)) \to 0
\end{multline*}
as $m \to + \infty$. Therefore the claim is true and $\omega = |\bar{k}|^2 > 0$ belongs to $\sigma_e(\cM)$. Since $\sigma_e(\cM)$ is closed and $\omega \in (0, +\infty)$ was chosen freely, we conclude the proof. 

\end{proof}

\subsection{Quasi-cylindrical domains}
We consider here quasi-cylindrical domains $\Omega$ in the sense of Definition \ref{def: classification}. 

\begin{theorem}\label{thm:essspecquasicyl}
Let $\Omega \subset \R^3$ be a quasi-cylindrical open set. Then $\sigma_e(\cM) \neq \emptyset$ and there exists a constant $C> 0$ such that $\min \sigma_e(\cM) \leq \frac{C}{R^2}$, where $R$ is the supremum among all the admissible radii for the sequence of disjoint balls $B(x_j, R)$ in the definition of quasi-cylindrical open set.
\end{theorem}
\begin{proof}
It is sufficient to show that $\omega_e = \min \sigma_e(\cM) < \infty$. Since $\cM$ is selfadjoint and semibounded from below, its essential numerical range $W_e(\cM)$ coincides with the convex hull of the extended essential spectrum of $\cM$, that is $W_e(\cM) = [\omega_e, +\infty)$. We will show that $\omega_e \leq \frac{\pi^2}{R^2}$ by constructing a sequence $(u_n)_n \subset X_{N0}(\Omega)$, $\norma{u_n}=1$, $u_n \rightharpoonup 0$ such that $\limsup_{n \to +\infty} \norma{\curl u_n}^2 \leq C/R^2$. \\
Let $\psi \in C^\infty_c(B(0,1)) \setminus \{0\}$. Define the vector field $w = \curl(\psi \be_3) = \nabla \psi \times \be_3$. Then $\norma{w} = \norma{\nabla_{\bar{x}} \psi}$, and $w \in H^1_0(B(0,1)^3 \cap X_{N0}(B(0,1))$ as $\Div w = \Div \curl(\psi \be_3) = 0$. Let $B(x_n, R)$, $n \in \N$, $x_n \to \infty$, be the sequence of balls in the definition of quasi-cylindrical open set $\Omega$. Define now
\[
w_n(x) = \frac{1}{R^{3/2}}
w\bigg(\frac{x - x_n}{R}\bigg),
\]
for $x \in B(x_n, R)$, and extended by zero outside $B(x_n, R)$ to the whole of $\Omega$. A direct computation shows that
\[
\norma{w_n}_\Omega = \norma{w}_\Omega, \quad \Div w_n = 0, \quad \norma{\curl w_n}^2_\Omega = \frac{\norma{\nabla_{\bar{x}} \p_3 \psi}_{B(0,1)}^2 + \norma{\Delta_{\bar{x}} \psi}_{B(0,1)}^2}{R^2},
\]
for all $n \in \N$. Define then $\tilde{w}_n = \frac{w_n}{\norma{w}_\Omega}$, $n \in \N$. Then $\norma{\tilde{w}_n} = 1$ and $\tilde{w}_n \rightharpoonup 0$ in $L^2(\Omega)^3$. We conclude that
\[
\omega_e = \min W_e(\cM) \leq \norma{\curl \tilde{w}_n}^2_\Omega = \frac{\norma{\nabla_{\bar{x}} \p_3 \psi}_{B(0,1)}^2 + \norma{\Delta_{\bar{x}} \psi}_{B(0,1)}^2}{\norma{\nabla_{\bar{x}} \psi}_{B(0,1)}^2 R^2} = \frac{C(\psi)}{R^2}.
\]

\end{proof}

\begin{theorem} \label{thm:essspecslablike}
Let $\Omega \subset \R^3$ be an open set containing a sequence of disjoint bounded open cuboids $(Q_n)_n$ such that there exists a rigid motion for which $Q_n = (0, a_n) \times (0, b_n) \times (0,c)$, with $a_n>0, b_n>0, c > 0$ for all $n$. We further assume that $a_n \to +\infty$ and $b_n \to + \infty$ as $n \to +\infty$. Then $\inf \sigma_e(\cM) \leq \pi^2/c^2$. 
\end{theorem}
\begin{proof}
We estimate $\omega_e = \inf \sigma_e(\cM)$ as in the proof of Theorem \ref{thm:essspecquasicyl}, by constructing a suitable sequence of test vector-fields as in the definition of $W_e(\cM)$. Let $\rho, \sigma \in C^\infty_c(0,1)$ be two smooth cut-off functions taking only nonnegative real values and being equal to $1$ in $(\frac13, \frac23)$, and let us assume that they have $L^2$-norm equal to 1, and that $\norma{\rho^{(j)}}^2_{L^2(0,1)} \simeq c_j \norma{\rho}_{L^2(0,1)}^2 = 1$, with an identical estimate for $\sigma$, $c_j$ being positive constants. Define the rescaled cut-off functions $\rho_n(x_1) = \rho(x_1/a_n)$ and $\sigma_n(x_2) = \sigma(x_2/b_n)$, for all $n \in \N$. Now define
\[
\varphi_n(x) = \rho_n(x_1) \sigma_n(x_2) \e^{\I (\xi \cdot \bar{x})} \sin(\pi x_3/c),
\]
for $x \in Q_n$, and finally set 
\[
u_n(x) = \curl (\varphi_n(x) \be_3) = \nabla \varphi_n \times \be_3,
\]
for all $x \in Q_n$, $n \in \N$. By construction, $u_n \in H^1_0(Q_n)^3$, as it has compact support in the variables $x_1$ and $x_2$, whereas at $x_3 = 0$ and $x_3 = c$ both its tangential and normal components are vanishing. Thus we can extend this vector field by zero to the whole of $\Omega$; we still call this extension $u_n$. We have therefore constructed a sequence of $H^1_0(\Omega)^3$ vector fields $u_n$ with support in $Q_n$, satisfying $\Div u_n = 0$ in $\Omega$ and $\nu \times u_n = 0$ on $\partial \Omega$ by construction. Since the cuboids $Q_n$ are disjoint, the support of $u_n$ escapes any compact subset of $\Omega$ as $n \to + \infty$. We now compute the $L^2$-norm of $u_n$. We have
\[
\begin{split}
&\norma{u_n}_{L^2(\Omega)^3}^2 = \norma{(\p_{x_1} \varphi_n, \p_{x_2} \varphi_n, 0)}^2_{L^2(\Omega)^3} \\
&= \norma{\sigma_n \sin(\pi x_3/c)(\rho_n' + \xi_1 \rho_n)}_{L^2(Q_n)}^2 + \norma{\rho_n  \sin(\pi x_3/c) (\sigma_n' + \xi_2 \sigma_n)}_{L^2(Q_n)}^2\\
& = C \bigg(\int_0^{a_n} (\rho_n'(x_1) + \xi_1 \rho_n(x_1))^2 \bigg)\bigg(\int_0^{b_n} \sigma_n(x_2)^2\bigg) \\
&\hspace{2cm} + C \bigg(\int_0^{a_n} \rho_n(x_1)^2 \bigg)\bigg(\int_0^{b_n} (\sigma_n' + \xi_2 \sigma_n)^2\bigg)
\end{split}
\]
where $C = \int_0^c \sin(\pi x_3/c)^2 = c/2$. Now we note that
\[
\int_0^{a_n} (\rho_n^{(j)}(x_1))^2 = c_j a_n^{1-2j}, \qquad \int_0^{b_n} (\sigma_n^{(j)}(x_1))^2 = c_j b_n^{1-2j}
\]
so that 
\begin{equation} \label{eq:L^2normun}
\norma{u_n}_{L^2(\Omega)^3}^2 = \int_\Omega |\xi|^2 \rho_n(x_1)^2 \sigma_n(x_2)^2 \sin(\pi x_3/c)^2 dx + \eps_n = A |\xi|^2 a_n b_n + o(a_n b_n),
\end{equation}
for a suitable constant $A > 0$, as $n \to +\infty$. It remains to compute $\norma{\curl u_n}^2$. Since $\curl u_n = \curl \curl (\varphi_n \be_3) = \nabla_{\bar{x}} \p_3 \varphi_n \be_{12} - \Delta_{\bar{x}} \varphi_n \be_3$, a direct computation shows that
\[
\curl u_n(x) = 
\begin{pmatrix}
&\I \xi_1 \rho_n \sigma_n \cos(\pi x_3/c) \pi/c + R_1(n, x) \\
&\I \xi_2 \rho_n \sigma_n \cos(\pi x_3/c) \pi/c + R_2(n, x) \\
& |\xi|^2 \varphi_n - R_3(n, x) 
\end{pmatrix}
\]
where $R_j$ are remainder terms that have $L^2$-norm much smaller than $a_n b_n$ as $n \to + \infty$. Since $\int_0^c |\cos(\pi x_3/c)|^2 dx_3 = \int_0^c |\sin(\pi x_3/c)|^2 dx_3 = C$ we conclude that 
\[
\begin{split}
\int_\Omega |\curl u_n|^2 dx &= |\xi|^4 \int_\Omega |\varphi_n|^2 + |\xi|^2 \frac{\pi^2}{c^2} \int_\Omega |\varphi_n|^2 + o(a_n b_n)\\
&= |\xi|^2 \bigg( |\xi|^2 + \frac{\pi^2}{c^2}\bigg) \int_\Omega |\varphi_n|^2 + o(a_n b_n) \\
&= \bigg( |\xi|^2 + \frac{\pi^2}{c^2}\bigg) \norma{u_n}^2_{L^2(\Omega)^3} + o(a_n b_n) 
\end{split}
\]
as $n \to + \infty$, from which we conclude that
\[
\frac{\norma{\curl u_n}_{L^2(\Omega)^3}^2}{\norma{u_n}^2_{L^2(\Omega)^3}} = \bigg( |\xi|^2 + \frac{\pi^2}{c^2}\bigg) + \frac{o(a_n b_n)}{A a_n b_n}
\]
hence $\frac{\norma{\curl u_n}_{L^2(\Omega)^3}^2}{\norma{u_n}^2_{L^2(\Omega)^3}} \to ( |\xi|^2 + \frac{\pi^2}{c^2})$ as $n \to +\infty$.\\
Since $\inf \sigma_e(\curl \curl) \leq \lim_{n \to +\infty} \frac{\norma{\curl u_n}^2}{\norma{u_n}^2} = |\xi|^2  + \pi^2/c^2$, and since $|\xi| > 0$ was arbitrary, we conclude the proof by letting $|\xi| \to 0^+$.
\end{proof}

\begin{rem}
Information on the growth of sequences of subdomains of $\Omega$ can only give upper bounds on the least essential spectral point of $\curl \curl$ in $\Omega$. In specific geometries where it is possible to lift Weyl sequences for the Laplacian on a lower dimensional section of $\Omega$ to Weyl sequences for $\curl\curl$ in the whole of $\Omega$, preserving the constraints $\Div u = 0$ in $\Omega$ and $\nu \times u = 0$ on $\partial \Omega$, more precise results can be proved (see e.g. the slab or the infinite cylinder).
\end{rem}

\begin{theorem}\label{thm:essspec-cylinders}
Let $\Omega \subset \R^3$ be an unbounded Lipschitz open set of $\R^3$ containing a sequence of cylinders $C_n = D \times (0, a_n)$ with $a_n \to + \infty$ and $D$ is a bounded domain of $\R^2$ admitting non-constant Neumann Laplacian eigenfunctions that are constant at the boundary. Then $\inf \sigma_e(\cM) \leq \mu_{1c}(D)$, where $\mu_{1c}(D)$ is the first Neumann Laplacian eigenvalue corresponding to a non-constant eigenfunction that is constant at the boundary. If $D$ is the disk of radius $R$, $\mu_{1c}(D) = \frac{j_{1,1}^2}{R^2} \simeq  \frac{14.682}{R^2}$, where $j_{1,1}$ is the first positive zero of the Bessel function $J_1$.
\end{theorem}
\begin{proof}
We will prove that $\inf \sigma_e(\cM) \leq \mu_{1c}(D)$ by exhibiting a Weyl sequence $(u_n)_n \subset X_{N0}(\Omega)$ with $\norma{u_n} = 1$, $n \in \N$, $u_n \rightharpoonup 0$ as $n \to +\infty$, and $\norma{\curl u_n}^2 \to \mu_{1c}(D)$ as $n \to +\infty$. Let $w$ be the solution to the following boundary value problem
\[
\begin{cases}
-\Delta w - \mu_{1c} w= 0, \quad &\textup{in $D$}, \\
w = \:\text{const}, \quad &\textup{on $\p D$}, \\
\frac{\p w}{\p \nu} = 0, \quad &\textup{on $\p D$},
\end{cases}
\]
A solution to such problem exists by assumption. We may assume without loss of generality that $\norma{w}_{L^2(D)} = 1$. Note that the gradient of $w$ is identically zero on the boundary $\p D$. Now define 
\[
u_n(\bar{x},x_3) = \curl ( w(\bar{x}) \rho_n(x_3) \be_3) = (\nabla_{\bar{x}} w \times \be_3) \rho_n(x_3)
\]
for all $n \in \N$, where $\rho_n(x_3) = \rho(x_3/a_n)$, and $\rho$ is the usual cut-off function in $(0,1)$ being equal to 1 in $(\frac13, \frac23)$. Note that both the tangential and the normal components of $u_n$ are vanishing on the boundary of the cylinder $C_n$. Extend then $u_n$ by zero to the whole of $\Omega$: the resulting vector field has null divergence in the whole of $\Omega$, it is identically zero on the boundary $\p \Omega$ and therefore it lies in $H_0(\curl, \Omega) \cap H(\Div0, \Omega)$. \\
It is clear that $\norma{u_n}^2_{L^2(\Omega)^3} = A \mu_{1c} a_n$ as $n \to +\infty$, and a direct computation as in the proof of Theorem \ref{thm:essspecslablike} shows that
\[
\begin{split}
\norma{\curl u_n}^2_{L^2(\Omega)^3} &= \int_\Omega |\nabla_{\bar{x}} w|^2 |\rho'_n|^2 dx + \int_\Omega |\Delta_{\bar{x}} w|^2 (\rho_n)^2 dx \\
&= o(a_n) + \mu_{1c}^2 \int_\Omega |w|^2 \rho_n^2 = A a_n \mu_{1c}^2 + o(a_n)
\end{split}
\]
and therefore
\[
\frac{\norma{\curl u_n}^2_{L^2(\Omega)^3}}{\norma{u_n}^2} \to \mu_{1c}
\]
as $n \to + \infty$. We conclude that $\inf \sigma_e (\cM) \leq \mu_{1c}$. 
\end{proof}

\begin{rem}
In view of the recent results by Colbrook and Stepaniants \cite{colbrook2026}, and Cao-Labora and de Dios Pont \cite{caolabora2026}, apart from the well-known case of the disk, there are many bounded domains in $\R^2$ admitting non-constant Neumann Laplacian eigenfunctions that are constant at the boundary; these domains are counterexamples to Schiffer's conjecture. Interestingly, they are not convex, even though they have a large number of axes of symmetry. 
\end{rem}

\section{Horn-shaped domains} \label{sec:filledhorn}
Let $\rho \in C^{1,1}(1, \infty)$ be a decreasing function with $\rho(z) \to 0$ and assume that
\[
{\rm ess} \liminf_{z\to +\infty} \: [\rho'(z)^2 - \rho(z) \rho''(z)] > - 1.
\] 
Let 
\begin{equation}\label{eq:Horn}
\Omega= \{ (\bar{x}, z) \in \R^2 \times (1, \infty) : |\bar{x}| < \rho(z) \},
\end{equation}
where $\bar{x} = (x_1, x_2)$. 

We will use the notation $r = |\bar{x}|$, $\be_r = (\cos \theta, \sin \theta, 0)$, $\be_\theta = (- \sin \theta, \cos \theta, 0)$ and $\be_z = (0,0,1)$, for $\theta \in [0, 2 \pi)$. Further define the lateral surface of $\Omega$ as
\[
\Gamma = \{ (\rho(z) \cos\theta, \rho(z) \sin\theta, z) : z > 1, \theta \in [0, 2 \pi) \}.
\]
For any $Z > 1$, define
\begin{equation}\label{eq:OmZ}
\Omega^Z = \Omega \cap \{ (\bar{x}, z) \in \R^3 : z > Z \}, \qquad S_z = \{ \bar{x} \in \R^2 : |\bar{x}| < \rho(z) \}.
\end{equation}

\subsection{Gaffney, Poincaré and Maxwell-Poincaré inequalities}
We first recall that in any $C^{1,1}$ bounded domain $\Omega$ of $\R^3$ with positive mean curvature, the strong Gaffney inequality \eqref{eq:strongGaff} holds, thus implying that $X_N(\Omega) \hookrightarrow H^1(\Omega)^3$.  Our aim will be to prove a rescaled asymptotic Maxwell-Gaffney inequality in the horn $\Omega$, that is 
\[
\norma{w}^2_{L^2(\Omega)^3} \leq C \rho(Z)^2 ( \norma{\curl w}^2_{L^2(\Omega)^3} + \norma{\Div w}^2_{L^2(\Omega)} ),
\]
for all $w \in X_N(\Omega)$ having support in $\Omega^Z$. 

The main fundamental observation is that the horn $\Omega$ has eventually positive mean curvature.
\begin{lemma}
The sum of the principal curvatures of the boundary of $\Omega$ is given by
\[
2 \cH(z) = \kappa_1(z) + \kappa_2(z) = \frac{1 + \rho'(z)^2 - \rho(z) \rho''(z)}{\rho(z) (1 + \rho'(z)^2)^{3/2}}
\]
for all $z > 1$.
\end{lemma}
\begin{proof}
We parametrise the lateral surface $\Gamma$ via $\Psi(\theta, z) = \rho(z)\be_r(\theta) + z \be_z$ for $\theta \in [0, 2 \pi)$, $z > 1$. The (non-normalized) tangent vectors to $\Gamma$ are given by
$\p_\theta \Psi = \rho \be_\theta$, $\p_z \Psi = \rho' \be_r + \be_z$, so that the exterior unit normal is
\[
\nu(\theta,z ) = \frac{\be_r - \rho'(z) \be_z}{\sqrt{1 + \rho'(z)^2}},
\]
$ \theta \in [0, 2 \pi)$, $z > 1$. 

Define $N(r, \theta, z) = \frac{\be_r - \rho'(z) \be_z}{\sqrt{1 + \rho'(z)^2}}$, for $ \theta \in [0, 2 \pi)$, $z > 1$ and $r$ near $\rho(z)$. Then 
\[
\Div N = \frac{1}{r} \p_r (r N_r) + \p_z N_z = \frac{1}{r \sqrt{1 + (\rho')^2}} - \frac{\rho''}{(1 + (\rho')^2)^{3/2}}
\]
Since $\Div N = \Div_{\partial \Omega} \nu = 2 \cH$ on $\partial \Omega$, the proof is complete.
\end{proof}

As a consequence of the previous lemma we see that as $z \to + \infty$, the mean curvature is positive, as $1 + \rho'(z)^2 - \rho''(z)\rho(z)$ is eventually positive.\\[0.1cm]

Before proceeding, let us note that: \\
(i) if $u\in H^1(\Omega)^3$, $u \times \nu = 0$ on $\p \Omega$ is equivalent to $u_\theta = 0$ and $u_z + \rho'(z)u_r = 0$ at $r = \rho(z)$. \\
(ii) if $u\in H^1(\Omega)^3$, $u(\cdot, z) \in H^1(S_z)^3$ for almost all $z \in (1, \infty)$, hence the trace of $u$ on $\partial S_z$ is well-defined for a.a. $z \in (1, \infty)$. 

\begin{prop} \label{prop:Poincare}
Let $D$ be the unit disk in $\R^2$. Let $P > 0$ be given. Then there exists a positive constant $C_p$ such that for all $p \in [-P, P]$ an every $u \in H^1(D)^3$ satisfying the boundary conditions
\begin{equation}\label{eq:bc}
\begin{cases}
u_\theta = 0, \quad &\textup{on $\p D$}, \\
u_z + p u_r = 0, \quad & \textup{on $\p D$}, 
\end{cases}
\end{equation}
the following inequality holds
\begin{equation}\label{eq:Poincare1}
\norma{u}_{L^2(D)^3} \leq C_p \norma{\nabla u}_{L^2(D)^3}.
\end{equation}
\end{prop}
\begin{proof}
Assume for a contradiction that there exists a sequence $(p_n)_n \subset [-P,P]$ and $(u_n)_n \subset H^1(D)^3$ satisfying \eqref{eq:bc} such that $\norma{u_n}_{L^2(D)^3} = 1$, $n \in \N$, and $\norma{\nabla u_n}_{L^2(D)^3} \to 0$ as $n \to + \infty$. Up to extracting a subsequence, we may assume that $p_n \to p \in [-P, P]$ as $n \to + \infty$. Define
\[
c_n = \frac{1}{\pi} \int_D u_n(x) \, dx, \qquad n \in \N.
\]
The standard Poincaré inequality implies that
\[
\norma{u_n - c_n}_{L^2(D)^3} \leq C_{p} \norma{\nabla u_n}_{L^2(D)^3} \to 0,
\]
as $n \to + \infty$. Since $|c_n| \leq \pi^{-1/2}$ for all $n$, up to a subsequence we may assume that $c_n \to c$. Thus, $u_n \to c =
(c^{(1)},c^{(2)},c^{(3)})$ strongly in $H^1(D)^3$ as $n \to +\infty$. By continuity of the trace map from $H^1(\p D)^3$ to $L^2(\p D)^3$, $u_n|_{\p D} \to c$ in $L^2(\p D)^3$, hence (possibly passing to a subsequence) almost everywhere on $\p D$. From \eqref{eq:bc} we conclude that
\[
\begin{cases}
- c^{(1)} \sin \theta + c^{(2)} \cos \theta = 0, \quad &\textup{for almost all $\theta \in [0, 2\pi)$}, \\
c^{(3)} + p (c^{(1)} \cos \theta + c^{(2)} \sin \theta) = 0, \quad &\textup{for almost all $\theta \in [0, 2\pi)$}.
\end{cases}
\]
The first equation implies that $c^{(1)} = c^{(2)} = 0$, and then the second equation entails $c^{(3)} = 0$. We conclude that $c = 0$, a contradiction to $\norma{u_n}_{L^2(D)^3} = 1$, $n \in \N$.
\end{proof}

\begin{corollary}
Let $D_R = B_{\R^2}(0, R)$. With $P$ and $p$ as in Proposition \ref{prop:Poincare}, there exists a constant $C_p > 0$ independent of $R$ such that
\begin{equation} \label{eq:rescaledPoinc}
\norma{u}_{L^2(D_R)^3}^2 \leq C_p^2 R^2 \norma{\nabla u}_{L^2(D_R)^3}^2,
\end{equation}
for all $u \in H^1(D_R)^3$ satisfying \eqref{eq:bc}.
\end{corollary}
\begin{proof}
Just apply \eqref{eq:Poincare1} to the rescaled vector field $\tilde{u} = u( R x) \in H^1(D)^3$, for all $u \in H^1(D_R)$ as in the statement. 
\end{proof}

We can now prove an asymptotic Maxwell-Gaffney inequality in the horn.
\begin{theorem}\label{thm:MaxGaff}
There exists $Z_0 > 1$ and $C > 0$ such that 
\begin{equation}\label{eq:MaxGaff}
\norma{w}^2_{L^2(\Omega)^3} \leq C \rho(Z)^2 \big( \norma{\curl w}^2_{L^2(\Omega)^3} + \norma{\Div w}^2_{L^2(\Omega)} \big),
\end{equation}
for all $Z \geq Z_0$ and every $w \in X_N(\Omega)$ with $\supp w \subset \Omega^Z$.
\end{theorem}
\begin{proof}
Choose $Z_0 > 1$ so that $\cH(z) > 0$ for all $z \geq Z_0$. 
The proof is in two steps. 
\vspace{2mm}

\noindent \textbf{Step 1.} Let $w$ have compact support in $\overline{\Omega} \cap \{z > Z\}$; then we can assume that $\supp w \subset \{x \in \overline{\Omega} : Z < z < Z_{\max} \}$. Let $G \subset \Omega$ be a bounded $C^{1,1}$ domain obtained in the following way: $\p G$ agrees with the lateral boundary of $\Omega$ on a neighbourhood of the $\p \Omega$ with nontrivial intersection with the support of $w$, and the remaining part of $\p G$ lies inside an open set on which $w$ = 0. Such a domain is obtained by cutting the horn at levels $z < Z$ and $z > Z_{\max}$ below and above the support of $w$, and rounding the two resulting corners entirely inside regions where $w = 0$.

Then the restriction of $w$ to $G$ belongs to $X_N(G)$, hence to $H^1(G)^3$ by the  standard Gaffney inequality \eqref{eq:Gaff} in bounded $C^{1,1}$-domains. Thus $w (\cdot, z) \in H^1(S_z)^3$ for almost all $z > Z$. By \eqref{eq:rescaledPoinc}, 
\[
\int_{S_z} |w(\bar{x}, z)|^2 \, d\bar{x} \leq C_p \rho(z)^2 \int_{S_z} |\nabla_{\bar{x}} w(\bar{x}, z)|^2\, d\bar{x}
\]
for almost all $z > 1$. By integrating in $z$, and by recalling that $\rho(z) \leq \rho(Z)$ for all $z$ in the support of $w$, 
\[
\norma{w}_{L^2(\Omega)^3}^2 \leq C_p \rho(Z)^2 \norma{\nabla_{\bar{x}} w}_{L^2(\Omega)^3}^2 \leq C_p \rho(Z)^2 \big( \norma{\curl w}^2_{L^2(\Omega)^3} + \norma{\Div w}^2_{L^2(\Omega)}   \big)
\]
concluding the proof in the case of compactly supported $w$.\\

\textbf{Step 2.} For the general case, choose a smooth cut-off function $\chi \in C_c^\infty[0, \infty)$, $0 \leq \chi \leq 1$, $\chi(s) = 1$ for $0 \leq s \leq 1$, $\chi(s) = 0 $ for $s \geq 2$ and define $\chi_R(z) = \chi(z/R)$, $w_R = \chi_R w$, for $R > 1$. Then $w$ has compact support in the $z$-direction, and it belongs to $X_N(\Omega)$ as 
\[
\curl(w_R) = \chi_R \curl w + \nabla \chi_R \times w, \qquad \Div w_R = \chi_R \Div w + \nabla \chi_R \cdot w,
\]
and $\nu \times w_R = \chi_R (\nu \times w) = 0$ on $\p \Omega$. Moreover,
\[
\norma{\nabla \chi_R \times w}_{L^2(\Omega)^3} \leq \frac{C}{R} \norma{w}_{L^2(\Omega)^3}, \qquad \norma{\nabla \chi_R \cdot w}_{L^2(\Omega)^3} \leq \frac{C}{R} \norma{w}_{L^2(\Omega)^3},
\]
both tending to zero as $R \to + \infty$. Thus, 
\[
w_R \to w, \quad \curl w_R \to \curl w, \quad \Div w_R \to \Div w
\]
as $R \to +\infty$. Step 1 implies that
\[
\norma{w_R}_{L^2(\Omega)^3}^2 \leq C_p \rho(Z)^2 \big( \norma{\curl w_R}^2_{L^2(\Omega)^3} + \norma{\Div w_R}^2_{L^2(\Omega)}   \big)
\]
with a constant $C_p$ independent of $R$. The limit as $R \to + \infty$ of the previous inequality yields the claim.
\end{proof}

\subsection{Essential spectrum and Zhislin essential spectrum for Maxwell and the relative Hodge Laplacian}
Recall the definitions of the essential spectrum $\sigma_e$, the essential numerical range $W_e$, 
the Zhislin essential spectrum $\sigma_{e, \Zhe}$ and the Zhislin essential numerical range
$W_{e, \Zhe}$ from subsection \ref{subsec:specth}. 
By definition $\sigma_{e, \Zhe}(\cdot) \subset \sigma_e(\cdot)$ and $W_{e, \Zhe}(\cdot) \subset W_e(\cdot) $. 
We recall that $\Delta_{\rm Hod}^{\rm rel}$ is the nonnegative selfadjoint operator associated with the quadratic form
\[
\mathfrak{t}[u] = \int_\Omega \big( |\curl u|^2 + |\Div u|^2 \big)\, dx, \quad u \in X_N(\Omega).
\]
As $\dom(\cM) \subset \dom(\Delta_{\rm Hod}^{\rm rel})$ one has that $\inf \sigma_e(\Delta_{\rm Hod}^{\rm rel}) \leq \inf \sigma_e(\cM)$. Moreover, $\inf \sigma_e(\Delta_{\rm Hod}^{\rm rel}) = \inf W_e(\Delta_{\rm Hod}^{\rm rel})$. In the next Proposition we show that $W_e(\Delta_{\rm Hod}^{\rm rel})$ can actually be replaced by $W_{e, \Zhe}(\Delta_{\rm Hod}^{\rm rel})$.
\begin{prop}\label{prop:essnumran}
$W_{e, \Zhe}(\Delta_{\rm Hod}^{\rm rel}) = W_e(\Delta_{\rm Hod}^{\rm rel}) $.
\end{prop}
\begin{proof}
To keep the notation more readable, we will write $\Delta$ in place of $\Delta_{\rm Hod}^{\rm rel}$. Due to Lemma \ref{lemma:convexZENR}, it is enough to show that $\inf W_e(\Delta) = \inf W_{e,\Zhe}(\Delta)$. \\
Fix $\omega \in W_e(\Delta)$; we need to establish $\omega \in W_{e, \Zhe}(\Delta)$. By definition of $W_e(\Delta)$ there exists $(u_n)_n \subset X_{N}(\Omega)$, $\norma{u_n} = 1$, $u_n \rightharpoonup 0$, and $\norma{\curl u_n}^2 + \norma{\Div u_n}^2 - \omega \to 0$ as $n \to +\infty$. We immediately note that $\norma{u_n}_{L^2(K)^3} \to 0$ in each precompact Lipschitz open subset of $\Omega$, as a consequence of the compactness of the embedding of $X_{N}(K)$ in $L^2(K)^3$ for all $K$. In particular, for each fixed $R> 0$, 
recalling the notation $\Omega^R$ (see \eqref{eq:OmZ}), we have $\norma{u_n}_{L^2(\Omega^R)^3} \to 1$ as $n \to + \infty$. \\[0.1cm]
Choose a smooth cut-off function $\chi \in C_c^\infty(0, \infty)$, $0 \leq \chi \leq 1$, $\chi(s) = 1$ for $s \geq 2 $, $\chi(s) = 0 $ for $0 \leq s \leq 1$ and define $\chi_R(z) = \chi(z/R)$. Further set $u_{n,R}= \chi_R u_n$, for $R > 1$. Then 
\[
\begin{split}
\curl u_{n,R} = \chi_R \curl u_n + \nabla \chi_R \times u_n, \\
\Div u_{n,R} = \chi_R \Div u_n + \nabla \chi_R \cdot u_n,
\end{split}
\]
for all $n$, and $\norma{u_{n,R}}_{L^2(\Omega)^3} = (1 - o(1)) \norma{u_n}_{L^2(\Omega)^3}$, as $n \to + \infty$. Select now a sequence of positive real numbers $(R_n)_n$, $R_n \to + \infty$, such that, up to taking a subsequence of $(u_n)_n$, $\norma{u_{n,R_n}}_{L^2(\Omega)^3} = (1 - o(1)) \norma{u_n}_{L^2(\Omega)^3}$ as $n \to + \infty$. Call this new sequence 
\[
w_n = u_{n, R_n}, \quad n \in \N.
\]
We now note that $|\nabla \chi_{R_n}| \leq C/R_n \to 0$ for all $x \in \Omega$, implying that the error terms $\nabla \chi_{R_n} \times u_n$ and $\nabla \chi_{R_n} \cdot u_n$ are vanishing as $n\to + \infty$. Therefore, $\curl w_n - \chi_{R_n} \curl u_n \to 0$ and $\Div w_n - \chi_{R_n} \Div u_n \to 0$ as $n \to +\infty$. We have proved that
\begin{itemize}
\item $\norma{w_n}_{L^2(\Omega)^3} = (1 - o(1)) \norma{u_n}_{L^2(\Omega)^3}$ as $n \to +\infty$; 
\item $\norma{\curl w_n}^2_{L^2(\Omega)^3} + \norma{\Div w_n}_{L^2(\Omega)^3}^2= \norma{\chi_{R_n} \curl u_n}^2_{L^2(\Omega)^3} + \norma{\chi_{R_n} \Div u_n}^2_{L^2(\Omega)^3} + o(1)$ as $n \to +\infty$; 
\end{itemize}
We conclude that 
\[
\begin{split}
\frac{\norma{\curl w_n}^2_{L^2(\Omega)^3} + \norma{\Div w_n}_{L^2(\Omega)^3}^2}{\norma{w_n}^2_{L^2(\Omega)^3}} &= \frac{\norma{\chi_{R_n}\curl u_n}^2_{L^2(\Omega)^3} + \norma{\chi_{R_n} \Div u_n}^2_{L^2(\Omega)^3}+ o(1)}{\norma{u_n \chi_{R_n}}^2_{L^2(\Omega)^3}} \\
&\leq \frac{\norma{\curl u_n}^2_{L^2(\Omega)^3} + \norma{\Div u_n}^2_{L^2(\Omega)^3} + o(1)}{(1 - o(1)) \norma{u_n}^2_{L^2(\Omega)^3}} \to \omega 
\end{split}
\]
By choosing $\omega = \inf W_e(\Delta)$, we conclude that
\[
\inf W_{e, \Zhe}(\Delta) \leq \inf W_e(\Delta);
\]
we already know that $W_{e, \Zhe}(\Delta) \subset W_e(\Delta)$, so now $\inf W_{e, \Zhe}(\Delta) = \inf W_e(\Delta)$. By Lemma \ref{lemma:convexZENR} the proof is complete.
\end{proof}

\subsection{Compactness of the resolvent}
Let $\Omega$ be as in \eqref{eq:Horn}, and $\omega_e = \min \sigma_e(\cM)$ in the Hilbert space $H(\Div 0, \Omega)$. 

\begin{theorem}\label{thm:compactres}
$\sigma_e(\cM) = \emptyset$, or equivalently, the resolvent of $\cM$ is compact as a bounded operator in $(H(\Div 0, \Omega), \norma{\cdot}_{L^2})$.
\end{theorem}
\begin{proof}
We have already observed that $\inf \sigma_e(\Delta) \leq \inf \sigma_e(\cM)$. By Prop. \ref{prop:essnumran}, $\inf \sigma_e(\Delta) = \inf W_{e,\Zhe}(\Delta)$. It is therefore sufficient to show that $ \inf W_{e,\Zhe}(\Delta) = +\infty$. To this end, assume for a contradiction that there exists $\omega \in \R$, $\omega = \inf W_{e,\Zhe}(\Delta)$. By definition of $W_{e,\Zhe}(\Delta)$ there exists a sequence $(u_n)_n \subset X_{N}(\Omega)$, $\norma{u_n} =1$, $\supp u_n \subset \Omega \cap B(0, R_n)^C$ for some $R_n \to + \infty$, such that
\[
\lim_{n \to +\infty} \big(\norma{\curl u_n}^2_{L^2(\Omega)^3} + \norma{\Div u_n}^2_{L^2(\Omega)^3}\big) = \omega.
\]
Theorem \ref{thm:MaxGaff} implies that
\[
\lambda(Z) = \inf_{\substack{u \in X_{N}(\Omega)\\ \supp u \subset \Omega^Z}} \frac{\norma{\curl u}^2_{L^2(\Omega)^3} + \norma{\Div u}^2_{L^2(\Omega)^3}}{\norma{u}_{L^2(\Omega)^3}}\to \infty
\]
as $Z \to + \infty$. In particular,
\[
\lim_{n\to+\infty} \frac{\norma{\curl u_n}^2_{L^2(\Omega)^3} + \norma{\Div u_n}^2_{L^2(\Omega)^3}}{\norma{u_n}_{L^2(\Omega)^3}} =\infty,
\]
a contradiction. This shows that $W_{e,\Zhe}(\Delta) = \emptyset$; equivalently, $ \inf W_{e,\Zhe}(\Delta) = +\infty$.

%
\end{proof} 

\begin{rem}
Note that Theorem \ref{thm:compactres} holds without assumptions on the volume $|\Omega|$, that is indeed allowed to be infinite. For instance, one can check that $\rho(z) = 1/z^\alpha$ is an admissible choice of radius function in \eqref{eq:Horn}. However, $|\Omega| < \infty$ if and only if $\alpha > 1/2$.
\end{rem}

\section{The perforated horn}\label{section:7}
The compactness of the resolvent proved in Theorem \ref{thm:compactres} may mislead one to think that if the transverse sections of a horn-shaped domain have area that vanish as $z \to + \infty$, then the resolvent of $\cM$ is compact. While this might be true for simply connected sections, the conjecture dramatically fails as soon as the sections have holes that extend at infinity. The aim of this section is to provide a general strategy to give information about the essential spectrum of $\cM$ in such perforated horns via a dimension reduction argument. \\[0.2cm]

Let $\rho, \delta \in C^2(1,\infty)$, be positive monotonically decreasing functions with $0 < \delta(z) < \rho(z)$, $z\in (1,\infty)$, and $\rho(z), \delta(z) \to 0$ as $z \to + \infty$. Define
\[
\Omega = \{(\bar{x},z) \in \R^3 : z > 1, \delta(z) < |\bar{x}| < \rho(z), \: z \in (1,\infty) \}.
\]

We recall here the following important core property, that we will use in the sequel. 
\begin{theorem}\label{thm:coreprop}
Define 
\[
\cD_N(\Omega) := \{ u \in C^\infty(\ov{\Omega})^3 : u \in X_N(\Omega), \, \exists Z > 1, \: u(\bar{x}, z) = 0, \: z \geq Z \}.
\]
Then $\overline{\cD_N(\Omega)}^{X_N(\Omega)} = X_N(\Omega)$, with respect to the norm 
$$\norma{\cdot}_{X_N(\Omega)} = \big(\norma{\cdot}^2_{L^2(\Omega)^3} + \norma{\curl \cdot}^2_{L^2(\Omega)^3} + \norma{\Div \cdot}^2_{L^2(\Omega)} \big)^{\frac12}$$
\end{theorem}
\begin{proof}
We give here just a sketch of the proof. If $u \in X_N(\Omega)$ and $\chi_M \in C^\infty_c([1,M+1])$, $0 \leq \chi_M \leq 1$, $\chi_M = 1$ in $[1,M)$ then $\chi_M u \in X_N(\Omega)$ and $\chi_M u \to u$ in $X_N(\Omega)$ as $M \to + \infty$. This proves that 
$$X_{N,c}(\Omega) := \{ u \in X_N(\Omega) : \supp u \:\: \text{is compact in the $z$-direction}\, \}$$ is dense in $X_N(\Omega)$. \\
Choose now a bounded smooth open subset $\Omega_M$ of $\Omega$ containing the support of $\chi_M u$. It is classical that there exists a sequence of functions $(u_{M,j})_j \subset \cD_N(\Omega)$ such that $u_{M,j} \to \chi_M u$ in $X_N(\Omega)$ as $j \to + \infty$. \\
Fix a number $\eps > 0$ and choose $M$ sufficiently big so that $\norma{\chi_M u - u}_{X_N(\Omega)} \leq \eps$. Then
\[
\norma{u_{M,j} - u}_{X_N(\Omega)} \leq \norma{u_{M,j} - \chi_M u}_{X_N(\Omega)} + \eps;
\]
letting $j \to +\infty$ shows that $\norma{u_{M,j} - u}_{X_N(\Omega)} \leq \eps$, concluding the proof.
\end{proof}

In the sequel we use the notation $r = |\bar{x}|$. Since $\Omega$ has an infinite hole, each section $S_z = \{r \in (0,+\infty) : \delta(z) < r < \rho(z) \}$ admits a harmonic function $h^z: S_z \to \R$ that is constant at the boundary. We represent the continuous family of harmonic functions $h^z$ indexed in $z = (1,\infty)$ as $q = q(r,z)$, $q(\cdot, z ) = h^z(\cdot)$ It is convenient to introduce here the quantity
\begin{equation}\label{def:M}
M(z) := \log \frac{\rho(z)}{\delta(z)}.
\end{equation}
An explicit computation yields that the function $q$ is given by
\begin{equation}\label{def:q}
q(r,z) = \frac{\log r - \log \delta(z)}{M(z)}, \quad z > 1, r \in S_z.
\end{equation}
Note that $q(\cdot, z)$ monotonically increasing in $r$ for each $z > 1$, and $q(\rho(z), z) = 1$, $q(\delta(z),z) = 0$, $z > 1$, so that $\nabla q$ is normal to $\p \Omega$ on the lateral boundary $\Gamma$ of $\Omega$.

\begin{prop}\label{prop:Ua}
Let $a \in C^\infty_c(1, \infty)$ and define
\[
U_a(r, \theta, z) = a(z) \sqrt{\frac{M(z)}{2 \pi}} \nabla q(r,z), \]
for $(r,\theta,z) \in (0,\infty) \times [0, 2\pi) \times (1, \infty) : \delta(z) < r < \rho(z)$.
Then $U_a$ has the following properties:
\begin{enumerate}[label=(\roman*)]
\item $U_a \in H_0(\curl, \Omega)$;
\item With $U_a = R_a + Z_a$, $R_a$ being the radial component and $Z_a$ the longitudinal component of $U_a$, we have
\begin{equation}
\norma{R_a}_{L^2(\Omega)^3}^2 = \norma{a}^2_{L^2(1,\infty)}, \qquad \Div R_a = 0.
\end{equation}
\item With
\begin{equation}\label{p}
p(z) = \frac{M'(z)}{2 M(z)}, \quad z > 1,
\end{equation}
we have 
\[
\curl U_a(r,z)= \frac{a'(z) + p(z) a(z) }{\sqrt{2\pi M(z)}} \frac{1}{r} \be_\theta, \qquad \delta(z) < r < \rho(z), z > 1,
\]
and 
\[
\curl\curl U_a(r,z) = \frac{- a''(z) - p'(z) a(z) + p^2(z) a(z)}{\sqrt{2\pi M(z)}}\frac{1}{r} \be_r
\]
for $\delta(z) < r < \rho(z)$, $z > 1$.
\end{enumerate}
\end{prop}
\begin{proof}
(i) Clearly $U_a$ is a smooth vector field with compact support in the $z$-direction, so it is enough to check that $\nu \times U_a = 0$ on $\Gamma$. But this is clear as $\nabla_{\Gamma} q = 0$, as $q$ is constant on $\Gamma$. Thus the tangential trace of $U_a$ is 0 on $\p\Omega$.\\[0.1cm]
(ii) Since $\partial_r q (r,z) = (M(z) r)^{-1}$, we immediately see that
\begin{equation}\label{eq:defRa}
R_a(r,z) = \frac{a(z)}{\sqrt{2 \pi M(z)}}\frac{1}{r} \be_r,
\end{equation}
and therefore 
\[
\norma{R_a}^2_{L^2(\Omega)^3} = \int_1^\infty \int_0^{2\pi} \int_{\delta(z)}^{\rho(z)} \frac{|a(z)|^2}{2 \pi M(z) r^2} r \; dr\,d\theta\,dz = \norma{a}^2_{L^2(1,\infty)},
\]
and
\[
\Div R_a(r,z) = \frac{1}{r} \p_r (r \cdot 1/r) \frac{a(z)}{\sqrt{2\pi M(z)}} = 0,
\]
identically for all $r,z$. \\[0.1cm]
(iii) We note that
\[
\begin{split}
\curl U_a &= \frac{1}{\sqrt{2\pi}} \nabla (a \sqrt{M}) \times \nabla q = \frac{1}{\sqrt{2\pi}} \bigg[\bigg(a' \sqrt{M} + \frac{M' a}{2\sqrt{M}}\bigg) \be_z\bigg] \times \nabla q \\
&= \frac{\p_r q}{\sqrt{2\pi}} \bigg[\bigg(a' \sqrt{M} + \frac{M' a}{2\sqrt{M}}\bigg)\bigg] \be_\theta = \frac{1}{\sqrt{2\pi}Mr} \bigg[\bigg(a' \sqrt{M} + \frac{M' a}{2\sqrt{M}}\bigg)\bigg] \be_\theta
\end{split}
\]
from which we deduce the formula in the statement. Similarly, for $\curl\curl U_a$ we have
\[
\begin{split}
\curl\curl U_a &=  - \p_z (\curl U_a \cdot \be_\theta) \be_r \\
&= - \bigg[ \frac{(a'' + p' a + p a') \sqrt{M} - (a' + p a ) \frac{M'}{2 \sqrt{M}}}{\sqrt{2\pi} M r} \bigg] \be_r \\
&= - \bigg[ \frac{ a'' + p' a + p a' - a'p + p^2 a}{\sqrt{2\pi M} r}\bigg] \be_r
\end{split}
\]
concluding the proof.
\end{proof}

\begin{definition} Define, initially on functions $a \in C^\infty_c(1,\infty)$, the operator 
\[
H_0 a = \bigg(\!-\!\frac{d}{dz} + p \bigg) \bigg( \frac{d}{dz} + p \bigg) a = - a'' - p' a + p^2 a.
\]
With $V = p^2 - p'$, the corresponding quadratic form is
\[
\mathfrak{b}_0[a] = \int_1^\infty (a' + p a)^2\, dz = \int_1^\infty ( (a')^2 + V a^2 )\, dz,  \quad \dom(\mathfrak{b}_0) = C^\infty_c(1,\infty).
\]
Define $\mathfrak{b}$ as the closure of $\mathfrak{b}_0$ in $L^2(1,\infty)$. Finally, define $H_V$ as the unique nonnegative selfadjoint operator associated with the quadratic form $\mathfrak{b}$, via the second representation theorem for forms. 
\end{definition}

\begin{rem}
Since $p \in L^2(1, Z)$ for each $Z > 1$, one can check that $\mathfrak{b}_0$ is closable in $L^2(1,\infty)$. Moreover, it is possible to prove an explicit description of the domain of $\mathfrak{b}$: $\dom(\mathfrak{b}) = \{ a \in L^2(1,\infty) : \: a \in H^1_{\rm loc}(1,\infty), \: a(1) = 0, \: a' + p a \in L^2(1,\infty) \:\}$. Clearly, if $p \notin L^2(1,\infty)$, in general $a' \notin L^2(1,\infty)$ and $p a \notin L^2(1,\infty)$. 
\end{rem}

Before proceeding we need some further information on $Z_a$, the axial component of $U_a$. We will use the following notation:
\[
\hat{\delta}(z) = \frac{\delta'(z)}{\delta(z)}, \quad \hat{\rho}(z) = \frac{\rho'(z)}{\rho(z)}, 
\]
and 
\[
\eta(z) = \rho(z) \bigg( \int_0^1 (s \hat{\delta}(z) + (1-s) \hat{\rho}(z) )^2 e^{-2sM(z)} \, ds  \bigg)^{\frac12},
\]
for $z \in (1,\infty)$.

\begin{prop}
With the notation introduced above, we have
\[
\p_z q = - \frac{(1-q) \hat{\delta} + q \hat{\rho}}{M},
\]
and for each fixed $a \in C^\infty_c(1,\infty)$, 
\[
\norma{Z_a}_{L^2(\Omega)^3}^2 = \int_1^\infty |a(z)|^2 \eta(z)^2 \, dz
\]
\end{prop}
\begin{proof}
Recall that by definition
\[
q(r,z) = \frac{\log r - \log \delta(z)}{M(z)}, \qquad \delta(z) < r < \rho(z), \, z > 1.
\]
Thus,
\[
\p_z q (r,z) = \frac{- \hat{\delta}(z) M(z) - (\log r - \log \delta(z)) M'(z)}{M(z)^2}.
\]
Since $M'(z) = \hat{\rho}(z) - \hat{\delta}(z)$, we conclude that
\[
\p_z q (r,z) =  \frac{- \hat{\delta}(z) (M(z) - q(r,z) M(z)) - q(r,z) M(z) \hat{\rho}(z) }{M(z)^2}
\]
from which we deduce the first formula in the statement. For the other one, we need to compute
\[
\norma{Z_a}^2_{L^2(\Omega)^3} = \int_1^\infty \int_0^{2\pi} \int_{\delta(z)}^{\rho(z)} a(z)^2 \frac{M(z)}{2 \pi} (\p_z q (r,z))^2 \,r\,dr\,d\theta\,dz.
\]
For brevity of notation we will not write the dependence of the several terms on the given variables, unless needed. Recalling the formula for $\p_z q$, and integrating in $\theta$, we obtain
\[
\norma{Z_a}^2_{L^2(\Omega)^3} = \int_1^\infty \int_{\delta(z)}^{\rho(z)} a^2 M\frac{ |(1-q) \hat{\delta} + q \hat{\rho}|^2}{M^2} \,r\,dr\,dz.
\]
Fix $z \in (1, \infty)$. We will now perform the change of variables $r \mapsto s = (1 - q(r, z))$ in the inner integral. Note that $q$ is monotonically increasing in $r$, and therefore the change of variables defines an admissible diffeomorphism. Recalling the notation $q(\cdot, z) = h^z(\cdot)$, we note that, for fixed $z \in (1,\infty)$, if $s = 1 - q(r,z)$, $r = (h^z)^{-1}(1-s) = \delta \e^{(1-s)M}$, and $\p_r s = - \p_r q = - (M r)^{-1} = - (M \delta \e^{(1-s)M})^{-1}$, thus implying
\[
\begin{split}
&\int_{\delta(z)}^{\rho(z)} a^2 M\frac{ |(1-q) \hat{\delta}(z) + q \hat{\rho}(z)|^2}{M^2} \,r\,dr \\
&= -\int_0^1 \frac{a^2}{M}(s \hat{\delta}(z) + (1-s) \hat{\rho}(z))^2 \delta \e^{(1-s)M} (- M \delta \e^{(1-s)M}) \, ds \\
&= \int_0^1 a^2 (s \hat{\delta}(z) + (1-s) \hat{\rho}(z))^2 \delta^2 \e^{2(1-s)M} \, ds
\end{split}
\]
Recalling that $\e^{2M} = \e^{2 \log(\rho/\delta)} = (\rho/\delta)^2$, we conclude the proof. 
\end{proof}

\begin{rem}\label{rem:decayofeta}
As $|s \hat{\delta} + (1-s) \hat{\rho}| \leq \max \{ |\hat{\delta}|, |\hat{\rho}| \}$ for every $s \in [0,1]$, we conclude that
\begin{equation}\label{eq:basicest}
\eta(z)^2 \leq \rho(z)^2 \max\{ |\hat{\delta}(z)|, |\hat{\rho}(z)| \}^2 \frac{1 - \e^{-2 M(z)}}{2 M(z)}, \qquad z>1.
\end{equation}
In practical cases, \eqref{eq:basicest} is enough to estimate the asymptotic behaviour of $\eta$. \\
Let us consider an example. Let $\alpha, \gamma > 0$ be two real parameters. If $\rho(z) = \e^{-z^\alpha}$, $\delta(z) = \e^{-z^\alpha - \e^{\gamma z}}$, then $M(z) = \e^{\gamma z}$, $|\hat{\delta}(z)| = \alpha z^{\alpha -1} + \gamma \e^{\gamma z}$ and $|\hat{\rho}(z)| = \alpha z^{\alpha -1}$, so
\[
\eta(z)^2 \leq C \e^{-2 z^\alpha} \gamma^2 \e^{2 \gamma z} \frac{1}{2 \e^{\gamma z}} \leq C \e^{ 2 (\gamma z - z^\alpha) }
\]
and the right hand side tends to zero if $\alpha > 1$; or, if $\alpha = 1$ and $\gamma < 1$. See also Example \ref{ex:superexpM} for a more detailed estimation of $\eta$ in this case. 
\end{rem}

As usual let $P_\nabla = \nabla (\Delta_\Omega^{\rm dir})^{-1} \Div$ be the projection onto $\nabla H^1_0(\Omega)$ and define 
\[
W_a = (I - P_\nabla) U_a = R_a + Z_a - P_\nabla Z_a = R_a + (I - P_\nabla) Z_a.
\]
Then
\begin{equation}\label{eq:WaRa}
\norma{W_a - R_a}_{L^2(\Omega)^3} \leq \norma{(I - P_\nabla) Z_a}_{L^2(\Omega)^3} \leq \norma{Z_a}_{L^2(\Omega)^3} \leq \sup_{z \in \supp a} |\eta(z)| \norma{a}_{L^2(1,\infty)}.
\end{equation}
Moreover, recalling Prop \ref{prop:Ua}, the notation \eqref{eq:defRa}, and the identity $\curl\curl W_a = \curl\curl U_a$, we conclude that
\begin{equation}\label{eq:1Dreduction}
(\curl \curl - \omega)W_a = R_{(H_V-\omega)a} - \omega (W_a - R_a)
\end{equation}
for all $a \in C^\infty_c(1,\infty)$. We then deduce the following theorem

\begin{theorem}\label{thm:Weylmodestransplant}
Assume that $\sup_{z \geq Z} |\eta(z)| \to 0$ as $Z \to + \infty$. Then 
\[
\sigma_{e,\Zhe}(H_V) = \sigma_e(H_V) \subseteq \sigma_e(\cM).
\]
\end{theorem}
\begin{proof}
If $\omega \in \sigma_{e,\Zhe}(H_V)$ and $(a_n)_n$ is a Weyl-Zhislin singular sequence for $H_V$, then $(\frac{W_{a_n}}{\norma{W_{a_n}}})_n$ is a Weyl sequence for $\cM$ at the same frequency $\omega$, in view of \eqref{eq:1Dreduction} and \eqref{eq:WaRa}.
\end{proof}

\begin{corollary}
Assume that $\sup_{z \geq Z} |\eta(z)| \to 0$ as $Z \to + \infty$ and that the potential 
\[
V(z) = p(z)^2 - p'(z) = \frac{1}{4} \bigg( \frac{M'(z)}{M(z)}\bigg)^2 - \frac{1}{2}\bigg( \frac{M'(z)}{M(z)} \bigg)',
\]
converges pointwise to zero as $z \to + \infty$. Then $\sigma_e(\cM) = [0, +\infty)$. 
\end{corollary}
\begin{proof}
The assumption $V(z) \to 0$ as $z \to + \infty$ implies that $\sigma_e(H_V) = [0,+\infty)$, hence the conclusion from Theorem \ref{thm:Weylmodestransplant}.
\end{proof}

\begin{example}
Let $c \in (0,1)$, $\alpha > 0$ and let 
\[
\Omega = \{ (\bar{x},z) : z > 1, \: c \, \Phi_\alpha(z) < |\bar{x}| < \Phi_\alpha(z) \},
\]
with $\Phi_\alpha(z) = \e^{-z^\alpha}$. Then $M(z) = \log(1/c)$, $|\hat{\delta}(z)| = |\hat{\rho}(z)| = \alpha |z|^{\alpha -1}$. According to \eqref{eq:basicest}, since $\rho$ is exponentially decaying and $\max\{|\hat{\delta}(z)|, |\hat{\rho}(z)|\}$ is at most polynomially growing, we conclude that $\sup_{z \geq Z} |\eta(z)| \to 0$ as $Z \to + \infty$, and $V = 0$ as $M$ is constant. Thus, $\sigma_e(\cM) = [0, + \infty)$, independently of the choice of $\alpha > 0$. 
\end{example}

\begin{example}
Let $\beta > 0$ and consider the perforated exponential horn
\[
\Omega = \{ (\bar{x},z) : z > 1, \: \e^{-z - z^\beta} < |\bar{x}| < \e^{-z} \},
\]
Then $M(z) = z^\beta$, $|\hat{\delta}(z)| = 1 + \beta z^{\beta - 1}$, $|\hat{\rho}(z)| = 1$, so $\sup_{z \geq Z} |\eta(z)| \to 0$ as $Z \to + \infty$. Moreover, 
\[
V(z) = \frac{\beta (\beta + 2)}{4 z^2} \to 0,
\]
as $z \to + \infty$. We conclude that $\sigma_e(\cM) = [0,+\infty)$.
\end{example}

\begin{example}\label{ex:superexpM}
Let $\gamma > 0$ be given and define
\[
\Omega = \{ (\bar{x},z) : z > 1, \: \e^{-z - \e^{2 \gamma z}} < |\bar{x}| < \e^{-z} \}.
\]
Then $M(z) = \e^{2 \gamma z}$, $|\hat{\delta}(z)| = 1 + 2 \gamma \e^{2 \gamma z}$, $|\hat{\rho}(z)| = 1$, and $p(z) = \gamma$, so that $V(z) = \gamma^2$. By direct computation, one checks that
\[
\eta(z)^2 = \e^{-2z} \int_0^1 |s (1 + 2 \gamma \e^{2 \gamma z}) + (1-s) |^2 \e^{-2s \e^{2\gamma z}}\, ds = \frac{\e^{-2z}}{\e^{2 \gamma z}} \int_0^{\e^{2 \gamma z}} (1 + 2 \gamma u)^2 \e^{-2u}\, du
\]
so that $\eta(z)^2 \leq C \e^{-2z} \to 0$ as $z \to + \infty$, independently of $\gamma$. However, since $V$ is constant in this case, Theorem \ref{thm:Weylmodestransplant} only implies that $[\gamma^2, + \infty) \subseteq \sigma_e(\cM)$. 
\end{example}

\section{The perforated horn with exponentially large $M$}\label{sec:finalsec}
Example \ref{ex:superexpM} suggests that in order to move the essential spectrum of $\curl\curl$ up in the positive half-line, it is not enough to make $\rho$ shrink faster; it is necessary that the potential $ V = p^2 - p'$ be strictly positive. However, to show that this is also sufficient one needs additional assumptions. We will restrict therefore ourselves to a toy model.

We will consider the perforated horn given by
\[
\Omega = \{ (\bar{x}, z) \in \R^3 : \e^{- z - \e^{2 \gamma z^\alpha}} < |\bar{x}| < \e^{-z}, z > 1 \},
\]
where $\alpha > 0$ is a fixed parameter. The aim of this section is to prove Theorem \ref{thm:sigmaessperforatedhornalpha}
%

The proof of this theorem requires several steps; each of those is the content of a subsection. For later use, we introduce the notations $\mathbb{S}^1 = \R/(2 \pi \Z)$ and $\Gamma = \p \Omega \setminus (\mathbb{S}^1 \times \{1\})$. 

\subsection{Step 1: cylindrical coordinates}
Let $C$ be the perforated cylinder $(1/2,1) \times \mathbb{S}^1 \times (1, \infty)$ with annular cross-section,  and define
\[
r(s,z) = \rho(z) \e^{-2(1-s) M(z)}, \qquad \frac12 < s < 1, z > 1.
\]
The function
\[
\Phi(s,\theta,z) = (r(s,z) \cos \theta, r(s,z) \sin \theta, z), \qquad (s, \theta, z) \in C
\]
defines a smooth diffeomorphism of $C$ onto $\Omega$. We carry forward the notations
\[
\hat{\delta}(z) = \frac{\delta'(z)}{\delta(z)}, \quad \hat{\rho}(z)= \frac{\rho'(z)}{\rho(z)}, \quad M(z) = \log \frac{\rho(z)}{\delta(z)},
\quad p(z) = \frac{M'(z)}{2M(z)} \]
from Section \ref{section:7} and define
\[ \beta(s,z) = 2(1-s) \hd(z) + (2s -1) \hr(z), \;\;\; B(s,z) = \frac{\beta(s,z)}{2M(z)}. \] 
A direct computation shows that
\[
\p_s r = 2 M r, \quad \p_z r = r ( \hr - (2-2s) \hr + (2-2s) \hd) = \beta r,
\]
from which we deduce the following.

\begin{lemma}
The pullback of the Euclidean metric $g = \Phi^*g_E$ is given by
\[
g = (2 M)^2 r^2 ds^2 + 4 M \beta r^2 ds dz + r^2 d\theta^2 + (1 + \beta^2 r^2) dz^2
\]
that is
\[
(g_{ij})_{ij} = 
\begin{pmatrix}
(2 M)^2 r^2 & 0   & 2 M \beta r^2 \\
0           & r^2 & 0             \\
2M \beta r^2 & 0 & 1 + \beta^2 r^2 
\end{pmatrix}
\]
with determinant $\det g = (2 M)^2 r^4$. The inverse metric is given by
\[
(g^{ij})_{ij} = \begin{pmatrix}
\frac{1}{(2M)^2 r^2} + (\frac{\beta}{2 M})^2& 0   & - \frac{\beta}{2M} \\
0                  & \frac{1}{r^2} & 0             \\
- \frac{\beta}{2M} & 0 & 1 
\end{pmatrix}
\]
\end{lemma}
\begin{proof}
With $\be_r = (\cos \theta, \sin \theta, 0)$, $\be_\theta = (-\sin \theta, \cos \theta, 0)$, $\be_z = (0,0,1)$, we have $\p_s \Phi = 2M r \be_r$, $\p_\theta \Phi = r \be_\theta$, $\p_z \Phi = \beta r \be_r + \be_z$, from which the claim follows after standard computations.
\end{proof}

\begin{lemma}\label{lemma:curldivcyl}
Let $E \in C^1(\Omega)^3$ be a smooth vector field such that $\nu \times E = 0$ on $\Gamma$, with representation in cylindrical coordinates $E = E_r \be_r + E_\theta \be_\theta + E_z \be_z$. Let $u = u_s ds + u_\theta d\theta + u_z dz$ denote the representation of $E$ in the coordinates $({s},\theta,z)$ induced by $\Phi$, that is, $u = \Phi^*(E)$ or equivalently, $(u_s, u_\theta, u_z)^T = (D\Phi^T(s,\theta,z))(E \circ \Phi)$. Then
\[
u_s = 2 M r (E_r \circ \Phi), \quad u_\theta = r (E_\theta \circ \Phi), \quad u_z = \beta r (E_r \circ \Phi) + (E_z \circ \Phi),
\]
and $u$ inherits the boundary conditions $u_\theta = 0$, $u_z = 0$ at $s \in \{1/2, 1\}$.
Additionally,
\begin{equation}\label{EL2}
\norma{E}^2_{L^2(\Omega)^3} = \int_C \bigg[ \frac{|u_s|^2}{2 M} + 2 M |u_\theta|^2 + 2 M r^2 \bigg| u_z - \frac{\beta}{2M} u_s\bigg|^2 \bigg] \, ds d\theta dz.
\end{equation}
Writing 
$c_{s\theta} u = \p_s u_\theta - \p_\theta u_s$, $c_{\theta z} u = \p_\theta u_z - \p_z u_\theta$, $c_{zs} u = \p_z u_s - \p_s u_z$, we further have
\begin{equation}\label{CEL2}
\norma{\curl E}^2_{L^2(\Omega)^3} = \int_C \bigg[ 2 M \bigg| c_{\theta z}u + \frac{\beta}{2M} c_{s \theta}u \bigg|^2 + \frac{|c_{zs}u|^2}{2 M} + \frac{|c_{s\theta}u|^2}{2 M r^2}\bigg] \, ds d\theta dz.
\end{equation}
Finally
\begin{equation}
\norma{\Div E}^2_{L^2(\Omega)} = \int_C \frac{1}{2M r^2} \big| \Theta(u) \big|^2 dsd\theta dz
\label{DEL2}
\end{equation}
in which 
\[
\Theta(u) := \bigg[ \p_s \bigg( \frac{u_s}{2M} - \beta r^2 \bigg( u_z - \frac{\beta}{2M}u_s \bigg) \bigg) + 2M\p_\theta u_\theta + \p_z \bigg( 2 M r^2 \bigg(u_z - \frac{\beta}{2M} u_s\bigg) \bigg) \bigg] 
\]
\end{lemma}
\begin{proof}
The proof of this lemma is a standard computation. We postpone it to the Appendix.
\end{proof}
\subsection{Step 2: renormalized vector fields}
Given $u \in L^2(C, g)$ define the renormalized vector field $U = (X, Y, W) \in L^2(C)^3$ by
\begin{equation}\label{def:rescaling}
X = \frac{u_s}{\sqrt{2M}}, \quad Y = \sqrt{2M} u_\theta, \quad Z = \sqrt{2M} r u_z, \quad W = \sqrt{2M} r \bigg( u_z - \frac{\beta}{2M} u_s \bigg).
\end{equation}
Then \eqref{EL2} becomes
\[
\norma{E}_{L^2(\Omega)^3}^2 = \norma{U}^2_{L^2(C)^3}.
\]
Thus, we have defined an isometric isomorphism $\cJ : L^2(\Omega)^3 \to L^2(C)^3$ mapping a given $E \in L^2(\Omega)^3$ to the rescaled vector field $U$ defined in \eqref{def:rescaling}. 
As a side remark, with
\[
\kappa(s,z) = \beta(s,z) r(s,z),
\]
the boundary condition $\nu \times E = 0$ on $\Gamma$ corresponds to $Y = 0$, $Z = 0$ and $W = - \kappa X$ at $s = 1/2, 1$.\\[0.1cm]

We will now compute $\curl U$ and $\Div U$ in relation with $\curl u$ and $\Div u$. {To this end} it is convenient to define the rescaled first-order operators
\begin{align} 
&K(X,Y) = \frac{ \p_s Y }{2 M r} - \frac{\p_\theta X}{r}, \label{def:K}\\
&L(X,Y) = \frac{ \p_s X }{2 M r} + \frac{\p_\theta Y}{r}, \label{def:L} \\
&T f = \frac{1}{2 M r} (\p_s f - 2 M f), \\ 
& G f = \frac{\p_\theta f}{r}, \\
& H f = (D_B - p) f, \qquad D_B := \p_z - B \p_s.
\end{align}
We note that $D_B r = \p_z r - B \p_s r = \beta r - B (2 M) r = 0$ and $\p_s \beta = 2 (\hr - \hd) = 2 M'$. 

\begin{prop}
With the notation introduced above, let $U = (X, Y, W) = \cJ(E)$, with $E \in H_0(\curl, \Omega) \cap H(\Div, \Omega)$. Then we have
\[
\norma{\curl E}_{L^2(\Omega)^3}^2 = \int_C \big( |K(X,Y)|^2 + |H X - T W|^2 + |G W - H Y|^2  \big) \, ds d\theta dz.
\]
and 
\[
\norma{\Div E}_{L^2(\Omega)^3}^2 = \int_C |L(X,Y) + HW|^2\, ds d\theta dz
\]
\end{prop}
\begin{proof}
To compute {$\curl E$}, one needs to note that
\begin{align*}
&c_{s\theta} u = \p_s (Y/\sqrt{2M}) - \p_\theta(\sqrt{2 M} X) = \sqrt{2M} r K(X,Y),\\
&c_{zs} u = \p_z(\sqrt{2 M} X) - \p_s (Z/(\sqrt{2M}r) = \sqrt{2M} ( \p_z X + p X - T Z), \\
& T(\kappa X) = 2 p X + B \p_s X, \\
& \p_z X + p X - T Z = \p_z X - B \p_s X - p X - T W = H X - T W, \\
&c_{\theta z} u = (\beta)^{-1} \kappa \sqrt{2M} K(X,Y), \quad
\sqrt{2M} \bigg( c_{\theta z} + \frac{\beta}{2M} \bigg) = - H Y + G W,
\end{align*}
from which the formula for the $L^2$ norm of $\curl E$ follows using \eqref{CEL2}. Regarding 
{$\Div E$}, one needs to note that
\[
\Theta(u)/(\sqrt{2M} r) = \frac{\p_s X}{2M r} + \frac{\p_\theta Y}{r} + \p_z W - B \p_s W - p W = L(X,Y) + H W
\]
and use \eqref{DEL2}.
\end{proof}

\begin{rem}
The operator $H$ is skew-symmetric on $H^1_{0, \Gamma}(C)$ as
\begin{equation}\label{eq:Hadj}
\begin{split}
(Hf, g) &= ((\p_z f - B \p_s f - p) f, g) = (f, - \p_z g + \p_s(B g) - p g) \\
&= ( f, - \p_z g + 2p \p_s g + B \p_s g - p g ) = (f, - H g),
\end{split}
\end{equation}
for all $f,g \in H^1_{0, \Gamma}(C)$.
\end{rem}
\, \\

\subsection{Step 3: sectional Hodge decomposition}
Let $z > 1$ be fixed in this subsection. Recall that we use the notation $S_z$ to denote the section of $\Omega$ given by $\{(r, \theta, z) : \delta(z) < r< \rho(z), \theta \in [0, 2 \pi) \}$. Let $\xi = u_s ds + u_\theta d\theta$ be a 1-form on $S_z$. Recalling \eqref{def:rescaling}, we deduce that
\[
\int_0^{2\pi} \int_{1/2}^1 \bigg( \bigg|\frac{u_s}{\sqrt{2M}}\bigg|^2 + \frac{|u_\theta|^2}{r^2} \bigg) r dr d\theta = \int_{1/2}^1 \int_0^{2\pi} \big( |X|^2 + |Y|^2) \, d\theta ds,
\]
and
\[
d_S \xi = (\p_r u_\theta - \p_\theta u_r) dr \wedge d\theta = \frac{1}{\sqrt{2M}} \bigg( \frac{\p_s Y}{2 M r} - \frac{\p_\theta X}{r}\bigg) dr \wedge d\theta = \frac{K(X,Y)}{\sqrt{2M}}dr \wedge d\theta,
\]
hence
\begin{equation}\label{eq:dSK}
\norma{d_S \xi}_{L^2(S_z)^2}^2 = \int_{1/2}^1\int_0^{2\pi} |K(X,Y)|^2\, d\theta ds;
\end{equation}
Similarly, with $\delta_S \xi = (\p_s u_s + \p_\theta u_\theta) $ one can compute
\begin{equation}\label{eq:dSL}
\norma{\delta_S \xi}_{L^2(S_z)^2}^2 = \int_{1/2}^1\int_0^{2\pi} |L(X,Y)|^2\, d\theta ds
\end{equation}

The usual Hodge decomposition:
\[
L^2 \Lambda^1(S_z) = d_S H^1_0(S_z) \oplus \cH^1_{\rm rel}(S_z) \oplus \star_S d_S H_{\perp}^1(S_z)
\]
implies that the only one-forms in the section $S_z$ compatible with the boundary condition $\nu \times E = 0$ on $\Gamma$ are : \\
1) TE modes given by $d_S \phi$, where $\phi = 0$ on $\p S_z$. These modes are represented by $(X, Y) = ( (2M)^{-1/2} \p_s \phi, \sqrt{2M} \p_\theta \phi)$. \\
2) TM modes given by $\star_S d_S \psi$, where $\p_s \psi = 0$ on $\p S_z$. These modes are represented by $(X, Y) = ( -\sqrt{2M} \p_\theta \psi, (2M)^{-1/2} \p_s \psi)$. \\
3) The TEM mode given by $d_S \log r = 2 M ds$. Its renormalized representative with unitary $L^2$ norm is $h = (\pi^{-1/2}, 0)$.\\[0.2cm]

Given a 1-form $\xi$ defined in $\Omega$, assume that its restriction to $S_z$ has renormalized representative $(X, Y)$. Define the averaging operator
\[
(\cA X)(z) = \frac{1}{\sqrt{\pi}} \int_{1/2}^1 \int_0^{2\pi} X(s,\theta,z) \, d\theta ds.
\]
Set $\square = (1/2,1) \times (0, 2 \pi)$. Define the projectors 
\begin{equation} \label{def:projectors}
P_h(X,Y) = \bigg( \frac{\cA X}{\sqrt{\pi}}, 0\bigg) \qquad  P_\perp(X,Y) = (X,Y) - P_h(X,Y),
\end{equation}
for all vector fields $(X,Y) \in L^2(\square)^2$. Let
\[
Q_{\square}[X, Y, Z] = \norma{K(X,Y)}^2_{\square} + \norma{L(X,Y)}^2_{\square} + \norma{TZ}^2_{\square} + \norma{G Z}_{\square}^2
\]
where
\[
\norma{\cdot}^2_{\square} = \int_{1/2}^1 \int_0^{2 \pi} |\cdot|^2\, ds \, d\theta
\]
is the usual $L^2$-norm. 

\begin{prop}
There exists $c_0 > 0$ independent of $z$ for $z \geq z_0$ such that 
\begin{equation} \label{eq:transverseTETMest}
Q_{\square}[X,Y,Z] \geq \frac{c_0}{\rho(z)^2} \big( \norma{P_\perp(X,Y)}_{\square}^2 + \norma{Z}_{\square}^2 \big)
\end{equation}
for all smooth vector fields $U = (X,Y,Z) \in C^1(C)^3 \cap L^2(C)^3$ with $Y = 0$ and $Z = 0$ at $s \in \{1/2,1\}$.
\end{prop}
\begin{proof}
Fix $z \geq z_0$. We will represent $\xi = (X,Y)$ at fixed $z$ by means of a suitable orthonormal basis of $L^2(S_z)^2$ provided by the Hodge decomposition in the section $S_z$. Let $(\phi_j)_{j \geq 1}$ be the orthonormal basis of $L^2(S_z)$ made of Dirichlet Laplacian eigenfunctions, that is $- \Delta_D \phi_j = \la_j^D \phi_j$, $j \geq 1$. Let $(\psi_k)_{k \geq 1}$ be the Neumann Laplacian eigenfunctions with $\psi_1$ the first non-constant eigenfunction. Then $(d_S \phi_j)_j \oplus (\star_S d_S \psi_k)_k \oplus \{h\}$ is an orthogonal basis of $L^2(S_z)^2$. Thus we can write
\[
\xi = \sum_j (\xi, d_S \phi_j) d_S \phi_j + \sum_k (\xi, \star_S d_S \psi_k) \star_S d_S \psi_k + (\xi, h) h,
\]
Thus
\[
\norma{\xi}^2 = \sum_j \la_j^D |(\xi, d_S \phi_j)|^2 + \sum_k \la_k^N |(\xi, \star_S d_S \psi_k)|^2 + |(\xi, h)|^2 \norma{h}^2.
\]
Furthermore
\[
\norma{d_S \xi}^2 = \sum_k (\la_k^N)^2 |(\xi, \star_S d_S \psi_k)|^2, \quad \norma{\delta_S \xi}^2 = \sum_j (\la_j^D)^2 |(\xi, d_S \phi_j)|^2
\]
Now we note that $\la_j^{\sharp}(S_z)= c^{\sharp} \rho(z)^{-2}$, for all $j \geq 1$, where $\sharp \in \{D, N\}$, as the Laplacian eigenvalues with either Neumann or Dirichlet boundary conditions rescale as $|S_z|^{-1} = \pi^{-1}\rho(z)^{-2} (1 - \frac{\delta(z)^2}{\rho(z)^2})^{-1}$. Recalling \eqref{eq:dSK}, \eqref{eq:dSL}, we conclude that
\[
\norma{K(X,Y)}_{\square}^2 + \norma{L(X,Y)}^2_{\square} \geq c \rho(z)^{-2} \norma{P_\perp(X,Y)}_{\square}^2.
\]

We need to consider now the longitudinal component $Z$. Define $v(r, \theta) = u_z = \frac{Z}{\sqrt{2M} r}$. Since $Z = 0$ at $s = 1/2, 1$, $v \in H^1_0(S_z)$ and 
\[
\p_r v = \frac{1}{\sqrt{2M}} \bigg( \frac{\p_r Z}{r} - \frac{Z}{r^2} \bigg) = \frac{1}{\sqrt{2M} r} \bigg( \frac{\p_s Z}{2 M r} - \frac{Z}{r}   \bigg) = \frac{TZ}{\sqrt{2M} r},
\]
\[
\frac{\p_\theta v}{r} = \frac{GZ}{\sqrt{2M} r}.
\]
Hence, recalling that $r dr = 2 M r^2 ds$, we have
\[
\int_{S_z} |v|^2 \, r dr d\theta = \int_{\square} |Z|^2\, dsd\theta, \;\;\;\;\;
\int_{S_z} |\nabla v|^2 \, r dr d\theta = \int_{\square} \big( |TZ|^2 + |GZ|^2 \big) \, ds d\theta.
\]
The Poincar\'{e} inequality in $S_z$ implies now that
\begin{align*}
 \int_{\square} \big( |TZ|^2 + |GZ|^2 \big) \, ds d\theta & = \int_{S_z} |\nabla v|^2 \, r dr d\theta \\
  &\geq \frac{C}{\rho(z)^2}\int_{S_z} |v|^2 \, r dr d\theta = \frac{C}{\rho(z)^2}\int_{\square} |Z|^2\, dsd\theta,
\end{align*}
which completes the proof.
\end{proof}

Up to this point, we did not use the specific choice of $\delta$ and $\rho$ in the definition of $\Omega$. In the following step instead we will use properties related to the particular choice of functions $\delta$ and $\rho$. For further use, it is then better to write down explicitly the value of the several geometric parameters in play. Recall that $\alpha, \gamma$ are positive parameters and that
\[
\rho(z) = \e^{-z}, \quad \delta(z) = \e^{-z - \e^{2 \gamma z^\alpha}}, \quad M(z) = \log(\rho(z)) - \log(\delta(z)) = \e^{2 \gamma z^\alpha}.
\]
Define now
\[
\zeta(s,z) := 2 (1-s) M(z)
\]
for $s \in (1/2,1)$, $z > 1$. Then we have
\[
p(z) = \frac{M'(z)}{2M(z)} = \gamma \alpha z^{\alpha -1}, \quad \hat{\rho} = -1, \quad \hat{\delta}(z) = - 1 - 2 \gamma \alpha z^{\alpha -1} M(z),
\]
\[
b(s,z) = -1 -2 \gamma \alpha z^{\alpha -1} \zeta(s,z), \quad \kappa(s,z) = - \rho(z)( 1 + 2 \gamma \alpha z^{\alpha -1} ) \e^{-\zeta(s,z)},
\]
and finally define 
\[
\omega(s,z) = (D_B \kappa)(s,z) = r(s,z) (D_B b)(s,z) = 2 \gamma \alpha r(s,z) (z^{\alpha -1}  - (\alpha -1) z^{\alpha -2} \zeta(s,z)).
\]
Note that, for any choice of $\gamma$ and $\alpha$ we have
\begin{align}
&\sup_{\substack{z > R \\ 1/2 < s < 1}} |\kappa(s,z)| \to 0, \label{eq:decayparameters} \\
&\sup_{\substack{z > R \\ 1/2 < s < 1}} |\omega(s,z)| \to 0, \label{eq:decayparameters2}\\
&\:\:\,\sup_{z > R } \rho^2(z) p(z)^2 \to 0, \label{eq:decayparameters3}
\end{align}
as $R \to + \infty$. 

Finally define for future usage the following forms: 
\begin{equation}\label{def:Q}
\begin{split}
&\cQ[X,Y,Z] = \hspace{-1mm} \int_C \hspace{-1mm}\big( |K(X,Y)|^2 + |H X - TZ + 2 p X + B \p_s X |^2 \\
&\hspace{1.8cm}+ |G Z - H Y - \beta \p_\theta X|^2  + |L(X,Y) + H Z - \kappa H X - \omega X |^2 \big) \, dsd\theta dz,\\
&\dom(Q) = \{(X,Y,Z) \in L^2(C)^3 : \: Y = Z = 0 \, {\rm at} \, s = 1/2, 1, \, Q(X,Y,Z) < \infty \: \}, \\[0.15cm]
&\cQ_0[X,Y,Z] = \int_C \big( |K(X,Y)|^2 + |H X - TZ |^2 + |GZ - HY|^2 \\
&\hspace{5.5cm} + |L(X,Y) + HZ|^2 \big) \, ds d\theta dz, \\
&\dom(Q_0) = \{ (X,Y,Z) \in L^2(C)^3 : Y = Z = 0 \, {\rm at} \, s = 1/2, 1, \, Q_0(X,Y,Z) < \infty \: \}, \\[0.15cm]
&q_p(\varphi) = \int_1^\infty |\varphi'(z) + p(z) \varphi(z)|^2 \, dz, \\
&\dom(q_p) = \{ \varphi \in L^2(1,\infty) \, : \, \varphi \in H^1_{\rm loc}(1, \infty), \varphi' + p \varphi \in L^2(1,\infty), \varphi(1) = 0\}.
\end{split}
\end{equation}

Since $H X + 2 p X + B \p_s X = \p_z X + p X$, $W = Z - \kappa X$, and $T(\kappa X) = 2pX + B\partial_s X$ we see immediately that   
\[
\begin{split}
\cQ[X,Y,Z] &= \int_C \big ( |K(X,Y)|^2 + |HX - TW|^2 + |G W - H Y|^2 \\
&\hspace{5.5cm}+ |L(X,Y) + HW|^2 \big) \, dsd\theta dz \\
&= \int_\Omega \big( |\curl E|^2 + |\Div E|^2 \big)\, dx
\end{split}
\]
where $E = \cJ^{-1}(U)$, $U =(X,Y, W)$.

\subsection{Step 4: Auxiliary lemmata}
We list here three technical lemmata that we will use in the sequel. Their proofs are postponed to the appendix.

\begin{lemma}\label{lemma:aux1}
Let $p\in C^1([R,\infty);\R)$, let $p\ge0$, and suppose that, for some
constant $0 \leq c <1$,
\begin{equation}
 p'(z)\le c p(z)^2,
 \qquad z\ge R.
 \label{eq:pprime-kappa}
\end{equation}
Let $a\in\dom q_p = \{f \in L^2(1, \infty) : f \in H^1_{\rm loc}(1, \infty), f' + p f \in L^2(1,\infty), f(1) = 0\}$, and assume that $a=0$ almost everywhere on $(1,R)$.  Then $pa\in L^2(R,\infty)$ and
\begin{equation}
 (1-c)\int_R^\infty p(z)^2|a(z)|^2\, dz
 \le q_p[a].
 \label{eq:tail-square-cutoff}
\end{equation}
\end{lemma}

\begin{lemma}\label{lemma:aux2}
With $\cQ_0$ as in \eqref{def:Q}, the following exact identity holds:
\[
\cQ_0[X,Y,Z] = \int_C \big( |K(X,Y)|^2 + |H X|^2 + |TZ |^2 + |GZ|^2 + |HY|^2 \hspace{3cm}\]
\[ \hspace{6cm}+ |L(X,Y)|^2 + |HZ|^2 \big) \, ds d\theta dz,
\]
for all smooth vector fields $U = (X,Y,Z)$ with $Y = 0$ and $Z = 0$ at $s = 1/2$ and $s = 1$.
\end{lemma}

\begin{lemma}\label{lemma:aux3}
Let $z > 1$ be fixed. There exists a constant $C = C(\alpha, \gamma)$ such that
\begin{equation}\label{eq:transversecurlest}
\norma{B \p_s X}^2_{\square} + \norma{\beta \p_\theta X}_{\square}^2 \leq C \rho(z)^2 ( 1 + (2 \gamma \alpha z^{\alpha -1})^2)(\norma{K(X,Y)}^2_{\square} + \norma{L(X,Y)}_{\square}^2)
\end{equation}
for all smooth vector fields $(X,Y)$ in $H^1(\square)^2$ with $Y = 0$ on $\p \square$. 
\end{lemma}

\subsection{Step 5: the approximate Maxwell inequality}
The aim of this final subsection is to prove a lower bound on the energy $\norma{\curl E}^2 + \norma{\Div E}^2$ of any $E \in X_N(\Omega)$, $\norma{E} = 1$, with support in the $z \geq R$ tail of the horn, possibly up to small errors in $R$. We begin with two useful estimates.

\begin{lemma} \label{lemma:estQQ0}
{Let $E \in \cD_N(\Omega)$ be any smooth vector field with support in $\{ x \in \Omega : z > R \}$ and let with $ U = \cJ(E) = (X,Y,W) = (X,Y, Z - \kappa X)$. The following statements hold.}
\begin{enumerate}[label=(\roman*)]
\item if $0 < \alpha < 1$ there exists a constant $C_1 > 0$ {independent of $E$} {and $R$} such that
$$\cQ_0[X,Y,Z] \leq C_1 (\cQ[X,Y,Z] + \norma{X}^2) = C_1 (\norma{\curl E}^2 + \norma{\Div E}^2 + \norma{E}^2)$$. \vspace{-0.3cm}
\item if $\alpha \geq 1$ there exists a constant $C_2 > 0$  {independent of $E$} {and $R$} such that 
$$\cQ_0[X,Y,Z] \leq C_2 (\cQ[X,Y,Z]) = C_2 (\norma{\curl E}^2 + \norma{\Div E}^2 ) $$. \vspace{-0.4cm}
\end{enumerate}

\end{lemma}
\begin{proof}
Set $e_1 = 2pX + B \p_s X$, $e_2 = - \beta \p_\theta X$ and $e_3 = - \omega X - \kappa H X$, and $\tilde{U} = (X,Y,Z)$. We have $\norma{HX - TZ}^2 \leq 2 \norma{HX - TZ + e_1}^2 + 2 \norma{e_1}^2$, and arguing in a similar way with the squares involving $e_2$ and $e_3$ we deduce that
\[
\cQ_0[\tilde{U}] \leq 2 \cQ[\tilde{U}]+ 2 (\norma{e_1}^2 + \norma{e_2}^2 + \norma{e_3}^3).
\]
We now note that, {in view of Lemma \ref{lemma:aux3} and the definitions of $e_1$ and $e_2$,
there exists a constant $C$ independent of $R$ such that} 
\[
(\norma{e_1}^2 + \norma{e_2}^2 ) \leq C \norma{p X}^2 + o_R(1) (\norma{K(X,Y)}^2 + \norma{L(X,Y)}^2).
\]
 Moreover,
\[
\norma{e_3}^2 \leq o_R(1) ( \norma{X}^2 + \norma{H X}^2)
\]
due to \eqref{eq:decayparameters}. With the help of Lemma \ref{lemma:aux2}, the $o_R(1)$-terms can be moved to the left-hand side and absorbed in the expression for $\cQ_0$, for sufficiently large $R$. Thus, we obtain, for some $\eps \in (0,1)$,
\begin{equation}\label{proof:estQ0}
(1 - \eps) \cQ_0[\tilde{U}] \leq 2 \cQ[\tilde{U}] + C \norma{p X}^2.
\end{equation}
If $\alpha < 1$, $p(z) = \gamma \alpha z^{\alpha - 1} \leq \gamma \alpha R^{\alpha -1}$ for $z > R$ from which we deduce $(i)$. \\
If instead $\alpha \geq 1 $, we go back to the estimate \eqref{proof:estQ0}. We estimate the term on the right involving $pX$. Write $pX = pP_h X + pP_\perp X$, where we recall that $P_h X = \pi^{-1/2} \cA X $. We first note that, due to Lemma \ref{lemma:aux1} (here we use the assumption $\alpha \geq 1$), we have
\[
q_p[\cA X] = \int_1^\infty |(\cA X)'(z) + p(z)(\cA X)(z)|^2 \, dz \geq (1-c) \int_1^\infty |p(z) (\cA X)(z)|^2 \, dz,
\]
(where we can choose $c = 0$ for $\alpha = 1$ and $c = 1/2$ for $\alpha > 1$). We now note explicitly that 
\[
(h, \p_z X + p X )_{L^2(\square)} = (\cA X)' + p \cA X,
\]
and, recalling that $Z = 0$ at $s= 1/2$ and $s = 1$, 
\[
\int_{1/2}^1 TZ \, ds = \int_{1/2}^1 \bigg(\frac{1}{2M r} \p_s Z - \frac{Z}{r} \bigg)\, ds = - \int_{1/2}^1 \p_s\bigg(\frac{1}{2M r}\bigg) Z\, ds - \int_{1/2}^1\frac{Z}{r} \, ds,
\]
and since $\p_s\big(\frac{1}{2M r}\big) = -\frac{1}{r}$ we conclude that $\int_{1/2}^1 TZ \, ds = 0$.  \\
Recalling that one of the terms in the definition of $\cQ[\tilde{U}]$ is $\norma{\p_z X + p X - TZ}^2$, we conclude that 
\begin{equation} \label{eq:estQqp}
\cQ[\tilde{U}] \geq \int_C |(\p_z X + p X - TZ)|^2 \geq q_p[\cA X]\geq \frac12 \int_1^\infty |p(z) (\cA X)(z)|^2 \, dz,
\end{equation}
where in the second-to-last inequality on the right we used the Cauchy-Schwarz inequality (as $\norma{h}_{L^2(\square)} = 1$)
\[
\begin{split}
\int_\square |(\p_z X + p X - TZ)|^2 ds d\theta &\geq |( \p_z X + p X - TZ, h)_{L^2(\square)}|^2 \\
&= \left\lvert \frac{1}{\sqrt{\pi}} \int_{\square} (\p_z X + p X) ds d\theta \right\rvert^2,
\end{split}
\]
which holds for almost all $z > 1$. It remains to bound the term in $P_\perp X$. We note that
\[
\int_C p^2 |P_\perp X|^2 \leq (\sup_{z > R} p(z)^2 \rho(z)^2) \int_C |P_\perp X|^2 \rho^{-2} \leq o_R(1) \cQ_0[\tilde{U}],
\]
where in the last inequality we used \eqref{eq:transverseTETMest}. Thus,
\[
(1 - \eps) \cQ_0[\tilde{U}] \leq 2 \cQ[\tilde{U}] + C \norma{p \cA X}^2 + C \norma{p P_\perp X}^2 \leq (2 + 2C) \cQ[\tilde{U}] + C o_R(1) \cQ_0[\tilde{U}],
\]
from which we deduce the claim $(ii)$. 
\end{proof}

\begin{corollary}\label{cor:weightedTETMMaxineq}
With $N(X,Y,Z) = \norma{P_\perp(X,Y)}_{L^2(C)^2}^2 + \norma{Z}^2_{L^2(C)}$, there exists a constant $C> 0$ such that
\[ 
N(X,Y,Z) \leq
\begin{cases}
C \rho(R)^2 \big( \norma{\curl E}^2 + \norma{\Div E}^2 + \norma{E}^2 \big), \quad &\textup{if $0 < \alpha < 1$,} \\
C \rho(R)^2 \big( \norma{\curl E}^2 + \norma{\Div E}^2 \big), \quad &\textup{if $\alpha \geq 1$,} \\
\end{cases}
\]
for all $E \in \cD_N(\Omega)$ with support in $\{z > R\}$, with $\cJ(E) = U = (X,Y,Z - \kappa X) \in C^\infty(C) \cap L^2(C)^3$.
\end{corollary}
\begin{proof}
Integrating inequality \eqref{eq:transverseTETMest} in $z > 1$ we obtain
\begin{multline*}
\int_1^\infty \big( \norma{K}^2_{\square} + \norma{L}^2_{\square} + \norma{TZ}^2_{\square} + \norma{GZ}^2_{\square} \big) dz \\
\geq c_0 \int_1^\infty \frac{1}{\rho(z)^2} \big(\norma{P_\perp(X,Y)}_{L^2(\square)^2}^2 + \norma{Z}^2_{L^2(\square)} \big)\, dz
\end{multline*}
that rearranges to give $N(X,Y,Z) \leq c_0^{-1} \rho(R)^2 \cQ_0(X,Y,Z) \leq C \rho(R)^2 \cQ(X,Y,Z)$, where in the last inequality we used Lemma \ref{lemma:estQQ0}.
\end{proof}

\begin{theorem}[Critical case]\label{thm:critical}
Let $\alpha = 1$. There exists a constant $C_\gamma > 0$ such that
\[
\norma{E}^2_{L^2(\Omega)^3} (1 - o_R(1))(\gamma^2 - C_\gamma \e^{-2R}) \leq \norma{\curl E}^2_{L^2(\Omega)^3} + \norma{\Div E}^2_{L^2(\Omega)},
\]
for all $E \in H_0(\curl, \Omega) \cap H(\Div, \Omega)$ with support in $\{ (\bar{x}, z ) \in \Omega : z \geq R \}$.
\end{theorem}
\begin{proof}
Assume that $E \in \cD_N(\Omega)$ with support in $\{ (\bar{x}, z ) \in \Omega : z \geq R \}$. First note that since $\alpha = 1$, $p(z) = \gamma$. Let $\cJ(E) = U = (X,Y, W) = (X, Y, Z - \kappa X)$ be the representation of $E$ in renormalized rescaled cylindrical coordinates. Lemma \ref{lemma:aux1} implies that $q_p[\cA X] \geq \gamma^2 \norma{\cA X}_{L^2(1,\infty)}$, and Corollary \ref{cor:weightedTETMMaxineq} gives
$$N(X,Y,Z) \leq C \rho(R)^2 \cQ(X,Y,Z) = C \e^{-2R} \big( \norma{\curl E}^2_{L^2(\Omega)^3} + \norma{\Div E}^2_{L^2(\Omega)} \big).$$
Recall that by the computations in \eqref{eq:decayparameters}, $\sup_{z > R} \sup_{s \in (1/2,1)} |\kappa(s,z)| \to 0$ as $R \to + \infty$; we then deduce that
\[
\begin{split}
\norma{E}^2_{L^2(\Omega)^3} &= \norma{U}^2_{L^2(C)^3} \leq \norma{\cA X}_{L^2(C)}^2 + \norma{P_\perp(X,Y)}_{L^2(C)^2} ^2 + \norma{W}_{L^2(C)}^2 \\
&\leq \frac{q_p(\cA X)}{\gamma^2} + N(X,Y,Z) + (\sup_{z > R} \sup_{s \in (1/2,1)} |\kappa(s,z)|^2) \norma{X}^2_{L^2(C)} \\
&\leq \bigg(\frac{1}{\gamma^2} + C \e^{-2R} \bigg) \cQ(X,Y,Z) + o_R(1) \norma{E}^2_{L^2(\Omega)^3}.
\end{split}
\]
Thus,
\[
 \cQ(X,Y,Z) \geq (1 -o_R(1)) \bigg(\frac{1}{\gamma^2} + C \e^{-2R} \bigg)^{-1} \hspace{-2mm}\norma{E}^2_{L^2(\Omega)^3} 
   = \frac{\gamma^2(1 -o_R(1))}{1 + C\gamma^2 \e^{-2R}} \norma{E}^2_{L^2(\Omega)^3},
\]
and this concludes the proof in the case $E \in \cD_N(\Omega)$. \\
If $E \in X_N(\Omega)$ with support in $\{ (\bar{x}, z ) \in \Omega : z \geq R \}$, by Theorem \ref{thm:coreprop} there exists a sequence $(E_n)_n \subset \cD_N(\Omega)$, $E_n \to E$ in $X_N(\Omega)$. We may assume without loss of generality that $\supp E_n \subset \{ (\bar{x}, z ) \in \Omega : z \geq R - \eps_n \}$ for a sequence $(\eps_n)_n$, $\eps_n \to 0$ as $n \to + \infty$. The previous part of the proof implies that 
\[
\norma{E_n}^2_{L^2(\Omega)^3} (1 - o_{R}(1))(\gamma^2 - C_\gamma \e^{-2R}) \leq \norma{\curl E_n}^2_{L^2(\Omega)^3} + \norma{\Div E_n}^2_{L^2(\Omega)},
\]
with a constant $C_\gamma$ independent of $n$. Passing to the limit as $n \to +\infty$ concludes the proof. 
\end{proof}

\begin{theorem}[Supercritical case]\label{thm:supercritical}
If $\alpha > 1$, there exists a constant $\Lambda(R) \to + \infty$ as $R \to + \infty$ such that
\[
\norma{E}^2_{L^2(\Omega)^3} \Lambda(R) \leq \norma{\curl E}^2_{L^2(\Omega)^3} + \norma{\Div E}^2_{L^2(\Omega)},
\]
for all $E \in H_0(\curl, \Omega) \cap H(\Div, \Omega)$ with support in $\{ (\bar{x}, z ) \in \Omega : z \geq R \}$. Moreover, $\Lambda(R) = O(R^{2(\alpha -1)})$ as $R \to + \infty$.
\end{theorem}
\begin{proof}
Note that $p(z) = \gamma \alpha z^{\alpha -1}$, $z > 1$. Let $E \in \cD_N(\Omega)$ with support in  $\{ (\bar{x}, z ) \in \Omega : z \geq R \}$, and let $U = (X,Y, W) = (X, Y, Z - \kappa X)$ be the representation of $E$ in renormalized rescaled cylindrical coordinates. Lemma \ref{lemma:aux1} implies that
\[
q_p[\cA X] \geq \frac12 \int_R^\infty p(z)^2 (\cA X(z))^2 \, dz \geq c R^{2 \alpha -2} \norma{\cA X}^2_{L^2(R, \infty)},
\]
hence $\norma{\cA X}^2_{L^2(R, \infty)} \leq \frac{1}{c R^{2 \alpha -2}} \cQ(X,Y,Z)$, in view of \eqref{eq:estQqp}. As in the critical case $\alpha =1 $ we have the estimate
\[
(1 -o_R(1)) \norma{E}^2_{L^2(\Omega)^3} \leq \norma{\cA X}_{L^2(C)}^2 + N(X,Y,Z),
\]
and the previous inequalities and Corollary \ref{cor:weightedTETMMaxineq} now imply
\[
(1 -o_R(1)) \norma{E}^2_{L^2(\Omega)^3} \leq \bigg(\frac{1}{c R^{2 \alpha -2}} + C \e^{-2R}\bigg) \cQ(X,Y,Z).
\]
This is exactly the claimed inequality for $E \in \cD_N(\Omega)$. The extension of the inequality to a generic $E \in H_0(\curl, \Omega) \cap H(\Div, \Omega)$ as in the statement follows as in the proof of the previous theorem.
\end{proof}

\subsection{Step 6: conclusion of the proof}
\begin{proof}[Proof of Theorem \ref{thm:sigmaessperforatedhornalpha}]
We have already showed that $\sigma_e(\cM) = [0, + \infty)$ if $ 0 < \alpha < 1$ and that $\sigma_e(\cM) \supset [\gamma^2, + \infty)$ if $\alpha = 1$, see Theorem \ref{thm:Weylmodestransplant} and Example \ref{ex:superexpM}.\\
We first prove that $\sigma_e(\cM) \subset [\gamma^2, + \infty)$. Assume, for a contradiction, that there exists $\la < \gamma^2$, $\la \in \sigma_e(\cM)$. By definition, there exists a Weyl singular sequence $(u_n)_n \subset \dom(\cM)$, $\norma{u_n}=1$, $u_n \rightharpoonup 0$ and $(\cM - \la) u_n\to 0$ in $L^2(\Omega)^3$. In particular,
\[
((\cM - \la) u_n, u_n) = \norma{\curl u_n}^2 - \la \to 0,
\]
as $n \to + \infty$, implying that $(u_n)_n$ is a uniformly bounded sequence in $X_{N0}(\Omega)$. We first note that the sequence $(u_n)_n$ has the following property:
\begin{equation}\label{proof:CP}
\textup{$u_n \to 0$ as $n \to  +\infty$ in $L^2(K)^3$, for all precompact $K\subseteq \Omega$ with $K \subset \ov{K} \subseteq \ov{\Omega}$.}
\end{equation}
This is a direct consequence of the compact embedding of $X_N(K)$ in $L^2(K)^3$ and of $u_n \rightharpoonup 0$ as $n \to + \infty$.\\
Fix now $R > R_0$, for $R_0 > 1$ sufficiently large. Choose $\chi_R \in C^\infty(\R)$ with $0 \leq \chi_R \leq 1$, $\chi_R = 0$ in $(-\infty, R)$, $\chi_R = 1$ in $(R+1, +\infty)$ and $\norma{\chi'_R}_\infty \leq C$, for some constant $C > 0$ independent of $R$. Define
\[
u_{n,R} = \chi_R u_n, \quad n \in \N.
\]
Then $u_{n,R} \in X_{N0}(\Omega)$, and it has support contained in $\Omega \cap \{z > R\}$. We further have:\\
(i) $\norma{u_{n,R}} = \norma{u_n}_{\Omega \cap \{z > R\}} \to 1$, as $n \to + \infty$, by \eqref{proof:CP}. \\
(ii) $\norma{\Div u_{n,R}}^2 = \norma{\nabla \chi_R \cdot u_{n,R}}^2 \to 0$ as $n \to + \infty$, since $\nabla \chi_R$ has compact support and therefore \eqref{proof:CP} applies. \\
(iii) $\norma{\curl u_{n,R}}^2 = \norma{\chi_R \curl u_n }^2 + \norma{\nabla \chi_R \times u_n}^2 + 2 \re (\chi_R \curl u_n, \nabla \chi_R \times u_n)$, and the last two summands tend to zero as $n \to + \infty$ again by the compactness of the support of $\nabla \chi_R$ and property \eqref{proof:CP}. \\
Property (iii) immediately implies that
\[
\liminf_{n \to + \infty} \norma{\curl u_{n,R}}^2 = \liminf_{n \to + \infty} \norma{\chi_R \curl u_n }^2 \leq \la.
\]
Theorem \ref{thm:critical} applied to $u_{n,R}$ implies that
\[
\norma{u_{n,R}}^2_{L^2(\Omega)^3} (1 - o_R(1))(\gamma^2 - C_\gamma \e^{-2R}) \leq \norma{\curl u_{n,R}}^2_{L^2(\Omega)^3} + \norma{\Div u_{n,R}}^2_{L^2(\Omega)},
\]
for all $n \in \N$, $R > R_0$. Taking the $\liminf$ as $n \to + \infty$ on both hand-sides, with the help of $(i)-(iii)$, we conclude that
\[
(1 - o_R(1))(\gamma^2 - C_\gamma \e^{-2R}) \leq \la.
\]
Finally, take the $\lim$ as $R \to + \infty$ to conclude that $\gamma^2 \leq \la$, a contradiction. This proves that $\sigma_e(\cM) =[\gamma^2, +\infty)$ in the critical case $\alpha =1$. \\[0.1cm]
We now consider the supercritical case $\alpha \geq 1$. Assume for a contradiction that there exists $\la \in [0, +\infty)$, $\la \in \sigma_e(\cM)$. If $(u_n)_n$ is the Weyl sequence at the frequency $\la$, arguing as in the previous part of the proof, the sequence $(u_{n,R})_n \subset X_{N0}(\Omega)$ has support in $\Omega \cap \{z > R\}$ and Theorem \ref{thm:supercritical} implies that
\[
\frac{\norma{\curl u_{n,R}}^2 + \norma{\Div u_{n,R}}^2}{\norma{u_{n,R}}^2} \geq \Lambda(R) (1- o_{R}(1)),
\]
taking the $\liminf_{n \to + \infty}$ we immediately find that $\la \geq \Lambda(R)(1 -o(R))$ and the right-hand side can be made arbitrarily large by increasing $R$. This contradiction concludes the proof. 

%
\end{proof}

\begin{rem} \label{rem:generality}
Even though Theorem \ref{thm:sigmaessperforatedhornalpha} has been stated and proved only for the specific choices $\delta(z) = \e^{-z - \e^{2 \gamma z^\alpha}}$ and $\rho(z) = \e^{-z}$, the same proof applies to the case of the perforated horn $\Omega_{\delta \rho} = \{(\bar{x}, z) \in \R^3 : \delta(z) < |\bar{x}| < \rho(z), z > 1 \}$, provided that \eqref{eq:decayparameters}--\eqref{eq:decayparameters3} hold. Obviously, in this more general case, the claim of the theorem becomes $\sigma_e(\cM) = \sigma_e(H_V)$, as in the statement of Theorem \ref{thm:intro:generalperfhorn}, and one has to define 
\[
\hat{\delta}(z) = \frac{\delta'(z)}{\delta(z)}, \quad \hat{\rho}(z)= \frac{\rho'(z)}{\rho(z)}, \quad M(z) = \log \frac{\rho(z)}{\delta(z)},
\quad p(z) = \frac{M'(z)}{2M(z)}, \]
\[ \beta(s,z) = 2(1-s) \hd(z) + (2s -1) \hr(z), \;\;\; B(s,z) = \frac{\beta(s,z)}{2M(z)}, \] 
and 
\[
\kappa(s,z) = \beta(s,z) r(s,z),  \quad \omega(s,z) = (D_B \kappa)(s,z) = \p_z \kappa(s,z) - B(s,z) \p_s \kappa(s,z).
\]

\end{rem}

\subsection*{Acknowledgements}
The authors would like to thank D. Buoso, L. Provenzano, A. Savo, and P. Siegl for useful discussions. They would like to further thank D. Krej{\v c}i{\v r}{\'i}k for providing references and insights regarding the essential spectrum of the Dirichlet Laplacian in unbounded domains, and Bruno Colbois for pointing out relevant references regarding the Hodge Laplacian on differential forms in Riemannian manifolds. This work was partially supported by a grant from the Simons Foundation to the second author. Both authors would also like to thank the Isaac Newton Institute  for Mathematical Sciences, Cambridge, for support and hospitality during the programmes `Geometric Spectral Theory and Applications' and (second author) `Operator Methods for Dynamical Systems', where some of the work on this paper was undertaken. This work was supported by EPSRC grant EP/Z000580/1. The first author is a member of INdAM (Istituto Nazionale di Alta Matematica) and GNAMPA (Gruppo Nazionale di Analisi Matematica, Probabilità e loro Applicazioni); he acknowledges the support by the project GNAMPA PROGETTI DI RICERCA 2026, CUP E53C25002010001.

\subsection*{AI usage disclosure}
The authors used ChatGPT 5.6 Sol to check the computations in Sections 5-6-7 and in the Appendix. All AI-assisted verifications were independently checked by the authors, who take full responsibility for the contents of the paper.

\bibliographystyle{abbrv}
\bibliography{EssentialMaxbib}

\section{Appendix}

\begin{proof}[Proof of Lemma \ref{lemma:curldivcyl}]
Recall that the exterior derivative $d$ commutes with diffeomorphic pullbacks; hence $d u = c_{s\theta} u\, ds \wedge d\theta + c_{\theta z} u\, d\theta \wedge dz + c_{zs} u\, dz \wedge ds$. Thus, in order to compute $\norma{\curl E}^2$ in the new coordinates it is enough to compute 
\[
|d u|^2_g = |c_{s\theta} u|^2 |ds \wedge d\theta|^2 + |c_{zs} u|^2 |dz \wedge d\theta|^2 + |c_{\theta z} u|^2 |d\theta \wedge dz|^2 \]
\[ \hspace{4cm} + 2 \re (c_{s \theta} u \overline{c_{\theta z}u}) \langle ds \wedge d\theta, d \theta \wedge dz \rangle
\]
To compute the inner product between two 2-forms $\omega_1 = \sum_{ij} \omega_1^{ij}\, d x_i \wedge dx_j$, $\omega_2 = \sum_{ij}\omega_2^{ij}\, d x_i \wedge dx_j$, recall the formula
\[
\langle \omega_1, \omega_2 \rangle_g = \sum \omega_1^{i_1 k_1} \overline{\omega_2^{i_2 k_2}} g^{i_1 k_1} g^{i_2 k_2}.
\]
Thus, 
\begin{align*}
&\langle ds \wedge d \theta, ds \wedge d \theta \rangle_g = g^{ss} g^{\theta \theta} = \frac{1}{r^2} \bigg( \frac{1}{(2 M)^2 r^2} + \frac{\beta^2}{(2M)^2} \bigg), \\
&\langle d \theta \wedge d z, d \theta \wedge d z \rangle_g = g^{\theta \theta} g^{z z} = \frac{1}{r^2}, \\
&\langle dz \wedge ds, dz \wedge ds \rangle_g = g^{zz} g^{ss} - (g^{zs})^2 = \frac{1}{(2M)^2 r^2}, \\
&\langle ds \wedge d\theta, d\theta \wedge dz \rangle_g = - g^{zs}g^{\theta\theta} = \frac{\beta}{2M r^2}, \\
&\langle ds \wedge d \theta, ds \wedge dz \rangle_g = \langle d\theta \wedge dz , dz \wedge ds \rangle_g = 0.
\end{align*}
We conclude that
\[
\begin{split}
|du|_g^2 \sqrt{g} &= \bigg( \frac{1}{(2 M)^2 r^2} + \frac{\beta^2}{(2M)^2} \bigg) |c_{s\theta}|^2 + (2M) |c_{\theta z}|^2 + \frac{|c_{zs}|^2}{2M} + 2 \re(\beta c_{s\theta} \ov{c_{\theta z}}) \\
&= 2M \bigg| c_{\theta z} + \frac{\beta}{2M} c_{s\theta} \bigg|^2 + \frac{|c_{zs}|^2}{2M} + \frac{|c_{s\theta}|^2}{2Mr},
\end{split}
\]
hence,
\[
\norma{\curl E}^2_\Omega = \int_C \bigg(2M \bigg| c_{\theta z} + \frac{\beta}{2M} c_{s\theta} \bigg|^2 + \frac{|c_{zs}|^2}{2M} + \frac{|c_{s\theta}|^2}{2Mr} \bigg)\, ds d\theta dz.
\]

For the divergence, we have $\Div_g u = \frac{1}{\sqrt{g}} \sum_i \p_i ( \sqrt{g} g^{ij} u_j)$. Since
\begin{align*}
&\sqrt{g} ( g^{ss} u_s + g^{sz} u_z) = \frac{u_s}{2M} - \beta r^2 \bigg( u_z - \frac{\beta}{2M} u_s \bigg), \\
&\sqrt{g} g^{\theta\theta}u_\theta = 2M u_\theta, \\
&\sqrt{g} \big( g^{zs} u_s + g^{zz} u_z \big) = 2M r^2 \bigg(u_z - \frac{\beta}{2M} u_s \bigg), 
\end{align*}
we deduce that
\[
\begin{split}
\Theta(u) &= \sum_i \p_i ( \sqrt{g} g^{ij} u_j) = \frac{1}{2M r^2} \bigg[ \p_s \bigg(\frac{u_s}{2M} - \beta r^2 \bigg( u_z - \frac{\beta}{2M} u_s \bigg) \bigg)\hspace{2cm}\\
&\hspace{4.5cm}+ \p_\theta (2M u_\theta) + \p_z \bigg( 2M r^2 \bigg(u_z - \frac{\beta}{2M} u_s \bigg)   \bigg) \bigg],
\end{split}
\]
from which we conclude that
\[
\norma{\Div E}^2_{L^2(\Omega)} = \int_C \frac{|\Theta(u)|^2}{g} \sqrt{g} \, ds d\theta dz = \int_C \frac{|\Theta(u)|^2}{2M r^2} \, ds d\theta dz.
\]

\end{proof}

\begin{proof}[Proof of Lemma \ref{lemma:aux1}]
Choose real cutoffs $\chi_M\in C_c^\infty((z_0,\infty))$, $M>R+1$, such that
\[
 0\le\chi_M\le1,\qquad
 \chi_M=1\ \text{on }[R,M],\qquad
 \chi_M=0\ \text{on }[M+1,\infty),\qquad
 \norma{\chi_M'}_\infty\le C,
\]
where the constant $C$ is independent of $M$.  Set $a_M=\chi_Ma$.  Note that
\[
 D_pa_M=\chi_MD_pa+\chi_M'a.
\]
Moreover,
\begin{align*}
 \norma{a_M-a}_{L^2}
 &\to0,\\
 \norma{D_p(a_M-a)}_{L^2}
 &\le
 \norma{(1-\chi_M)D_pa}_{L^2}
 +\norma{\chi_M'a}_{L^2}
 \to0.
\end{align*}
Here $D_pa=0$ almost everywhere on $(1,R)$, since the support of $a$ is in $\{z > R\}$. Consequently
\begin{equation}
 q_p[a_M]\to q_p[a],
 \label{eq:graph-convergence-aN}
\end{equation}
as $M \to + \infty$.

For fixed $M$, the function $a_M$ has bounded support.  It belongs to $H^1$ on that bounded interval because $p$ is locally bounded and $a_M'=D_pa_M-pa_M\in L^2$.  Hence $p|a_M|^2\in W^{1,1}$ there and has compact support.  By integration by parts, it follows that
\begin{align}
 q_p[a_M]
 &=\int_{1}^\infty
 \left(|a_M'|^2+p^2|a_M|^2
       +p\,(|a_M|^2)'\right)\rd z \notag\\
 &=\int_{1}^\infty
 \left(|a_M'|^2+(p^2-p')|a_M|^2\right)\rd z \notag\\
 &\ge(1-c)
 \int_R^\infty p^2\chi_M^2|a|^2\rd z.
 \label{eq:compact-square-aN}
\end{align}
Fatou's lemma, \eqref{eq:graph-convergence-aN}, and the pointwise convergence $\chi_M\to1$ yield
\[
 (1-c)\int_R^\infty p^2|a|^2\rd z
 \le
 \liminf_{M\to\infty}q_p[a_M]
 =q_p[a].
\]
This both proves \eqref{eq:tail-square-cutoff} and shows that $pa\in L^2(R,\infty)$.
\end{proof}

\vspace{1cm}

\begin{proof}[Proof of Lemma \ref{lemma:aux2}]
Expand the last three squares in the definition of $\cQ_0$ \eqref{def:Q}:
\begin{align}
 \norma{HX-TZ}^2
 &=\norma{HX}^2+\norma{TZ}^2
   -2\operatorname{Re}\ip{HX}{TZ},                                \label{eq:expand1}\\
 \norma{-HY+GZ}^2
 &=\norma{HY}^2+\norma{GZ}^2
   -2\operatorname{Re}\ip{HY}{GZ},                                \label{eq:expand2}\\
 \norma{L(X,Y)+HZ}^2
 &=\norma{L(X,Y)}^2+\norma{HZ}^2
   +2\operatorname{Re}\ip{L(X,Y)}{HZ}.                                  \label{eq:expand3}
\end{align}
Thus, in order to prove Lemma \ref{lemma:aux2}, it is sufficient to establish the identity
\begin{equation}
 \ip {L(X,Y)}{HZ}=\ip{HX}{TZ}+\ip{HY}{GZ}.
 \label{eq:crossidentity}
\end{equation}
We prove \eqref{eq:crossidentity} in three steps.\\
\textbf{Step 1.} We claim that the following integration-by-parts formula holds: for each fixed $z > 1$, 
\begin{equation}
 \ip{L(X,Y)}\phi_{\square}
 =-\ip X{T\phi}_{\square}-\ip Y{G\phi}_{\square}
 +\int_{\mathbb S^1}
   \left[\frac{X\overline \phi}{2M r}\right]_{1/2}^{1}\dd\theta.
 \label{eq:Lparts}
\end{equation}
for arbitrary smooth $X,Y,\phi$, periodic in $\theta$. \\
Indeed, integrating $L$ against $\phi$ gives
\begin{align}
 \ip{L(X,Y)}\phi_{\square}
 &=\int\left(\frac1{2M r}\p_s X+\frac1r \p_\theta Y\right)
          \overline \phi\,\dd s\dd\theta                             \notag\\
 &=\int_{\mathbb S^1}
   \left[\frac{X\overline \phi}{2M r}\right]_{1/2}^{1}\dd\theta
 -\int X\,\partial_s\left(\frac{\overline \phi}{2M r}\right)
        \dd s\dd\theta
 -\int\frac1rY\,\overline{\p_\theta \phi}\,\dd s\dd\theta.
 \label{eq:Lparts1}
\end{align}
Since
\[
 \partial_s\left(\frac1{2M r}\right)
 =-\frac1r,
\]
we have
\[
 \partial_s\left(\frac{\overline \phi}{2M r}\right)
 =\frac1{2M r}(\overline{\p_s \phi}-2M\overline \phi)
 =\overline{T\phi}.
\]
Since $r$ is independent of $\theta$, 
\[
 \frac1r\overline{\p_\theta \phi}=\overline{G\phi}.
\]
Thus the claim has been proved. Note that if $\phi=Z$ and $Z=0$ on the two faces, the boundary term vanishes:
\begin{equation}
 \ip{L(X,Y)}Z_{\square}=-\ip X{TZ}_{\square}-\ip Y{GZ}_{\square}.
 \label{eq:LpartsZ}
\end{equation}

\noindent \textbf{Step 2.} We now claim that 
\[
 T(HZ)=H(TZ),\qquad G(HZ)=H(GZ).
\]
To prove this, we first note that since $D_Br=0$,
\[
 D_B\left(\frac1r\right)=0.
\]
Moreover
\[
 [D_B,\partial_s]f
 =D_B(\p_s f)-\partial_s(D_Bf)=\p_s B \p_sf = 2p \p_s f.
\]
Since $D_B(M)=M'=2pM$,
\[
 D_B\left(\frac1{2M r}\right)
 =-\frac{M'}{M}\frac1{2M r}
 =-\frac{p}{M r}.
\]
Thus
\begin{align*}
 D_B(Tf)
 &=D_B\left[\frac1{2M r}(\p_s f-2M f)\right]\\
 &=-\frac{2p}{2M r}(\p_s f-2M f)
   +\frac1{2M r}\bigl(\p_s(D_Bf)+2p\p_s f-2M' f-2M D_Bf\bigr)\\
 &=\frac1{2M r}\bigl(\p_s(D_Bf)-2M D_Bf\bigr)=T(D_Bf).
\end{align*}
Because $p$ is independent of $s$, $T(pf)=pTf$, and hence
\begin{equation}
 [H,T]=0.
 \label{eq:HTcomm}
\end{equation}
Similarly $D_B$ commutes with $r^{-1}\partial_\theta$, and $p$ is
independent of $\theta$, so
\begin{equation}
 [H,G]=0.
 \label{eq:HGcomm}
\end{equation}

\noindent \textbf{Step 3}.  We now establish \eqref{eq:crossidentity}. Apply first \eqref{eq:Lparts} with $\phi=HZ$, and integrate in $z$:
\begin{align}
 \ip L{HZ}
 =-\ip X{T(HZ)}-\ip Y{G(HZ)}
 +\int_{\mathbb S^1\times(1,\infty)}
   \left[\frac{X\overline{HZ}}{2M r}\right]_{1/2}^{1}
   \dd\theta\dd z.
 \label{eq:LHZ1}
\end{align}
The commutators \eqref{eq:HTcomm}--\eqref{eq:HGcomm} give
\[
 T(HZ)=H(TZ),\qquad G(HZ)=H(GZ).
\]
Recall (see \eqref{eq:Hadj}) that for smooth functions compactly supported in $z$, the following formal adjoint formula holds 
\begin{align}
 \ip{Hf}{g}+\ip f{Hg}
 &=
 -\int_{\mathbb S^1\times(1,\infty)}
 [Bf\overline g]_{s=1/2}^{s=1}\,\dd\theta\dd z.
 \label{eq:Hadjoint}
\end{align}

Use \eqref{eq:Hadjoint} twice to deduce that
\begin{align}
 -\ip X{H(TZ)}
 &=\ip{HX}{TZ}
 +\int_{\mathbb S^1\times(1,\infty)}
   [BX\overline{TZ}]_{1/2}^{1}\dd\theta\dd z,                      \label{eq:horizontalX}\\
 -\ip Y{H(GZ)}
 &=\ip{HY}{GZ}
 +\int_{\mathbb S^1\times(1,\infty)}
   [BY\overline{GZ}]_{1/2}^{1}\dd\theta\dd z.                      \label{eq:horizontalY}
\end{align}
The boundary term in \eqref{eq:horizontalY} is zero because $Y=0$. Since $Z=0$ on a fixed face, its tangential derivatives at that face vanish:
\[
\p_z Z=0,\qquad \p_\theta Z=0
 \quad\text{on }s=1/2,1.
\]
Therefore, on either face,
\begin{equation}
 HZ=\p_z Z-B\p_s Z-pZ=-B\p_s Z,\qquad
 TZ=\frac1{2M r}(\p_s Z-2M Z)=\frac{\p_s Z}{2M r}.
 \label{eq:boundaryHZTZ}
\end{equation}
The two remaining boundary integrands are consequently opposite:
\[
 BX\overline{TZ}
 +\frac{X\overline{HZ}}{2M r}
 =\frac{BX\overline{\p_s Z}}{2M r}
  -\frac{BX\overline{\p_s Z}}{2M r}=0.
\]
Substituting \eqref{eq:horizontalX}--\eqref{eq:horizontalY} into \eqref{eq:LHZ1} proves \eqref{eq:crossidentity}.  
\end{proof}

\, \\

\begin{proof}[Proof of Lemma \ref{lemma:aux3}]
Fix $z > 1$ and apply the change of variables $s \mapsto \zeta(s,z) = 2(1-s)M(z)$, $s \in (1/2,1)$.  Then
\begin{equation}\label{proof:changecoordszeta}
 \partial_s=-2M\partial_\zeta,\qquad
 \dd s=\frac{\dd \zeta}{2M},\qquad
 r=\rho\e^{-\zeta}.
\end{equation}
Recall that $1+ 2 \gamma \alpha z^{\alpha -1} \zeta = - \beta$.  Thus, in particular,
\begin{equation}\label{eq:BpsX}
 B \p_s X = \frac{\beta}{2M}(-(2M) \p_\zeta X) = -\beta \p_\zeta X.
\end{equation}
Due to the periodicity in $\theta$ of $X$ and $Y$, a Fourier series argument in the angular variable $\theta$ gives
\[
X(\zeta, \theta) = \sum_{m \in \Z} X_m(\zeta) \e^{\I m \theta}, \qquad Y(\zeta, \theta) = \sum_{m \in \Z} Y_m(\zeta) \e^{\I m \theta},
\]
$\zeta \in (0, M(z))$, $\theta \in [0, 2 \pi)$. Equations \eqref{def:K}, \eqref{def:L} then entail
\[
 K_m=-\frac1r(Y_m'+imX_m),\qquad
 L_m=\frac1r(-X_m'+imY_m),
\]
$m \in \Z$, where $'$ now means $\partial_\zeta$. 

For $m=0$, $L_0=-X_0'/r$, and therefore
\[
 |\beta X_0'|=|\beta r||L_0|
 \le C\rho(1+2 \gamma \alpha z^{\alpha -1})|L_0|.
\]

Let then $m\ge1$.  Define 
\[
 F=X_m+iY_m,\qquad G=X_m-iY_m.
\]
Then
\[
 L_m+iK_m=\frac1r(-F'+mF),\qquad
 L_m-iK_m=\frac1r(-G'-mG).
\]
Since $Y_m(\zeta)=0$ for $\zeta \in \{0,M\}$,
\[
 F(0)=G(0),\qquad F(M)=G(M).
\]
Write $f=-F'+mF$, $g=-G'-mG$; then
\[
|f|^2 = |F'|^2 + m^2 |F|^2 - m (|F|^2)', \quad |g|^2 = |G'|^2 + m^2 |G|^2 - m (|G|^2)'.
\]
Multiply the previous identities by $w(\zeta)=\e^{\zeta/2}$ and integrate by parts to obtain
\begin{align*}
 \int_0^M w(|f|^2+|g|^2)\dd \zeta
 ={}&\int_0^M w\bigl(
 |F'|^2+|G'|^2+(m^2+m/2)|F|^2\\
 &\hspace{35mm} +(m^2-m/2)|G|^2\bigr)\dd \zeta.
\end{align*}
The boundary contribution vanishes as $m[w(|G|^2-|F|^2)]_0^M=0$.  Since $m^2-m/2\ge m^2/2$ for $m \geq 1$,
\[
 \int_0^M w\left(
 |F'|^2+|G'|^2+\frac{m^2}{2}(|F|^2+|G|^2)\right)\dd \zeta
 \le\int_0^M w(|f|^2+|g|^2)\dd \zeta.
\]
Now
\[
 \beta^2 = (1+2 \gamma \alpha z^{\alpha -1} \zeta)^2 \le 2(1 + (2 \gamma \alpha z^{\alpha -1} \zeta)^2)\e^{-\zeta/2} e^{\zeta/2}
 \le 2(1+(2 \gamma \alpha z^{\alpha -1})^2)\e^{\zeta/2},
\]
and $X_m=(F+G)/2$, so
\[
 |X_m'|^2\le\frac12(|F'|^2+|G'|^2),\qquad
 m^2|X_m|^2\le\frac12m^2(|F|^2+|G|^2).
\]
Hence,
\begin{equation}\label{proof:Xm}
\begin{split}
 \int_0^M \beta^2(|X_m'|^2+m^2|X_m|^2)\dd \zeta &\le 2(1+(2 \gamma \alpha z^{\alpha -1})^2)\int_0^M\e^{\zeta/2}(|f|^2+|g|^2)\dd \zeta \\
 &\le 2(1+(2 \gamma \alpha z^{\alpha -1})^2)\int_0^M\e^{2 \zeta}(|f|^2+|g|^2)\dd \zeta 
 \end{split}
\end{equation}
But
\[
 |f|^2+|g|^2=2r^2(|K_m|^2+|L_m|^2),
\]
and recalling that $r = \rho e^{-\zeta}$, we have 
\begin{equation}\label{proof:estfinal}
\int_0^M\e^{2 \zeta}(|f|^2+|g|^2)\dd \zeta = \int_0^M \frac{2\rho^2}{2r^2}(|f|^2+|g|^2)\dd \zeta =  2\rho^2\int_0^M(|K_m|^2+|L_m|^2)\dd \zeta.
\end{equation}
By \eqref{proof:Xm} and \eqref{proof:estfinal} we obtain
\[
 \int_0^M (|\beta X_m'|^2 + |\beta \I m X_m|^2)\dd \zeta \leq 2(1+(2 \gamma \alpha z^{\alpha -1})^2)2\rho^2\int_0^M(|K_m|^2+|L_m|^2)\dd \zeta
\]
Recalling \eqref{proof:changecoordszeta}, \eqref{eq:BpsX} we have
\[
 \int_{1/2}^1(|B \p_s X_m|^2 + |\beta \I m X_m|^2)\dd \zeta \leq 2(1+(2 \gamma \alpha z^{\alpha -1})^2)2\rho^2\int_{1/2}^1(|K_m|^2+|L_m|^2)\dd \zeta
\]
for all $m \geq 0$. The case $m < 0$ is identical to the case $m > 0$ after replacing $m$ by $|m|$. Summation in $m$, and Parseval's identity prove \eqref{eq:transversecurlest}. 
\end{proof}

\end{document}